\documentclass[11pt]{amsart}
\usepackage{hyperref}
\usepackage{amsfonts}
\usepackage{amsmath}
\usepackage{amssymb}
\usepackage{amscd}
\usepackage{bm}
\usepackage{graphicx}
\usepackage{enumitem}
\usepackage{latexsym}
\usepackage{color}
\usepackage{comment}
\usepackage{epsfig}
\usepackage{amsopn}
\usepackage{xypic}
\usepackage{mathrsfs}
\usepackage{array}
\usepackage[usenames,dvipsnames]{pstricks}
\usepackage{epsfig}
\usepackage{pst-grad} 
\usepackage{pst-plot} 
\usepackage[space]{grffile} 
\usepackage{etoolbox} 
\usepackage{tikz-cd}
\makeatletter 
\patchcmd\Gread@eps{\@inputcheck#1 }{\@inputcheck"#1"\relax}{}{}
\makeatother

\hypersetup{pdfpagemode=UseNone}
\usepackage{booktabs}
\usepackage{longtable}
\theoremstyle{plain}
\newtheorem{lemma}{Lemma}[section]
\newtheorem*{theorem*}{Theorem}
\newtheorem*{lemma*}{Lemma}
\newtheorem*{proposition*}{Proposition}
\newtheorem*{conjecture*}{Conjecture}
\newtheorem*{corollary*}{Corollary}
\newtheorem*{problem*}{Problem}
\newtheorem{theorem}[lemma]{Theorem}

\newtheorem{corollary}[lemma]{Corollary}
\newtheorem{proposition}[lemma]{Proposition}

\theoremstyle{definition}
\newtheorem{definition}[lemma]{Definition}
\newtheorem{example}[lemma]{Example}
\newtheorem{remark}[lemma]{Remark}

\newcommand{\fto}[1]{\stackrel{#1}{\to}}

\newcommand{\CC}{\mathbb{C}}

\newcommand{\OO}{\mathcal{O}}

\newcommand{\id}{\mathrm{id}}

\newcommand{\gr}{\mathrm{gr}}

\newcommand{\mumax}{\mu_{\max}}
\newcommand{\mumin}{\mu_{\min}}

\newcommand{\Quadric}{\PP^1\times\PP^1}

\newcommand{\PP}{\mathbb{P}}

\DeclareMathOperator{\Hom}{Hom}

\DeclareMathOperator{\rk}{rk}
\DeclareMathOperator{\length}{length}
\DeclareMathOperator{\coker}{coker}

\DeclareMathOperator{\Ext}{Ext}
\DeclareMathOperator{\ext}{ext}

\DeclareMathOperator{\ev}{ev}

\DeclareMathOperator{\sHom}{\mathcal{H}\kern -.5pt\mathit{om}}
\DeclareMathOperator{\sTor}{\mathcal{T}\kern -1.5pt\mathit{or}}

\newcommand{\leqor}{\underset{{\scriptscriptstyle (}-{\scriptscriptstyle )}}{<}}

\begin{document}

\author[Neelarnab Raha]{Neelarnab Raha}
\address{Department of Mathematics, The Pennsylvania State University, University Park, PA 16802}
\email{neelraha@psu.edu}

\date{\today}

\subjclass[2020]{Primary: 14J60, 14J26. Secondary: 14D20}
\keywords{Vector bundles, moduli spaces, Clifford's theorem, Brill-Noether theory, smooth quadric}

\title{Higher rank Clifford's theorem on the smooth quadric}

\begin{abstract}
	Brill-Noether theory of curves has played a crucial role in the study of curves and their moduli since the 19th century, and has been extensively studied by several authors. Clifford's theorem provides a starting point in determining the emptiness of Brill-Noether loci by providing an upper bound on $h^0(L)$ for a line bundle $L$ on a smooth curve $C$ in terms of the degree of $L$. It also characterizes the cases for which equality holds.
	
	In this paper, we prove an analogous result for higher rank sheaves on $\Quadric$. Depending on how nice the first Chern class is, and whether the sheaf has global generation properties, we prove sharp upper bounds on $h^0(E)$ for slope semistable sheaves $E$ in terms of $\rk(E)$ and $c_1(E)$. We also find that any $E$ achieving the bound is a twist of a Steiner-like bundle, or closely related to such a bundle. As part of our investigation, we show that general extensions of stable vector bundles on elliptic curves and del Pezzo surfaces are semistable.
\end{abstract}

\dedicatory{To my Dad, for his unconditional love and support}

\maketitle

\setcounter{tocdepth}{1}
\tableofcontents

\section{Introduction}

Brill-Noether theory has played a significant role in the study of curves and their moduli since the 19th century. It has been extensively studied by several authors, both for line bundles as well as higher rank vector bundles on a general curve $C$. See \cite{CoskunHuizengaNuer24, GM, Newstead} for surveys. Also see \cite{GriffithsHarris, FultonLazarsfeld, Gieseker3, EisenbudHarris}.

On the other hand, Brill-Noether theory for moduli spaces of sheaves on surfaces is not as well understood, and has been the subject of several recent studies. Some interesting questions involve determining the nonemptiness, irreducibility and dimensions of the Brill-Noether loci. See \cite{GottscheHirschowitz, CoskunHuizengaWBN, CH-BrillNoeth_Hirzebruch, CoskunHuizengaNuer24, CoskunNuerYoshioka, CoskunNuerYoshioka2, LevineZhang, GLL23, CHR25}.

Clifford's theorem (see \cite[Chapter III]{ACGH}) gives an upper bound on $h^0(L)$ for a line bundle $L$ on a smooth curve $C$ only in terms of the degree of $L$, and characterizes the cases for which equality holds. This gives one a starting point in determining emptiness of Brill-Noether loci. Coskun, Huizenga and the author (see \cite{CHR25}) have recently proved an analog of Clifford's theorem for higher rank sheaves on $\PP^2$. They obtain a sharp upper bound on $h^0(E)$ for slope-semistable sheaves on $\PP^2$ in terms of the rank and the slope of the sheaf, and show that twisted \emph{Steiner bundles} are the ones that achieve the bound. They further discuss the structure of sheaves for which $h^0$ is close to the bound, and also study the geometry of the corresponding Brill-Noether loci.

Motivated by this, we prove an analog of Clifford's theorem for higher rank sheaves on the smooth quadric surface $\Quadric\subset\PP^3$. Analogous to the case of $\PP^2$, we find that twisted \emph{Steiner-like} bundles are the ones that have the maximal possible number of sections.

The author is also working on similar questions on other surfaces like Hirzebruch surfaces and del Pezzo surfaces for the purpose of his PhD thesis.

\subsection{Upper bounds on $h^0$}

Throughout, we consider $\Quadric$ with the polarization $$H:=\OO_{\Quadric}(1,1).$$ For any torsion-free sheaf $E$ of rank $r$ on $\Quadric$ with $c_1(E)=(a,b)$, the slope of $E$ is $$\mu_H(E):=\frac{c_1(E)\cdot H}{rH^2}=\frac{a+b}{2r}.$$ More generally, the \emph{slope associated to} any given rank $r\geq1$ and first Chern class $c_1$ is $\mu:=\dfrac{c_1\cdot H}{rH^2}$.

\begin{definition}
	Let $r\geq1$ be an integer, and let $\mu\geq-1$ be a rational number. Then we define $$\alpha_{\mu}:=\lfloor \mu\rfloor + 1\quad\mbox{ and }\beta_{r,\mu}:=r\alpha_{\mu}(2\mu-\alpha_{\mu}+2).$$
\end{definition}

\subsubsection{General bound}

We have the following general upper bound on $h^0(E)$ for torsion-free sheaves on $\Quadric$. It is an analog of Clifford's theorem for line bundles on curves.

\begin{theorem}\label{SummaryThm01}
	(Theorem \ref{thm:bound_sections}.) Let $E$ be a torsion-free sheaf of rank $r\geq1$ on $\Quadric$ with $\mu_{\max}(E)\geq-1$. Then $h^0(E)\leq\beta_{r,\mumax(E)}$.
\end{theorem}

An immediate corollary of this is the emptiness of Brill-Noether loci. Recall that the \emph{Brill-Noether locus} $B^k$ is defined to be the closure of the locus of stable sheaves having at least $k$ linearly independent global sections.

\begin{corollary}
	Let $M=M_{\Quadric,H}(r,c_1,\chi)$ be a moduli space of $H$-Gieseker semistable sheaves on $\Quadric$. The Brill-Noether locus $B^k\subset M$ is empty if $k>\beta_{r,\mu}$, where $\mu$ is the associated slope.
\end{corollary}

The bound above works well for sheaves with \emph{balanced} first Chern class, i.e., those for which $c_1=(a,b)$ with $a$ and $b$ close to each other. For example, the bound is sharp for line bundles of the form $\OO_{\Quadric}(a,a)$ and $\OO_{\Quadric}(a,a\pm1)$.

We investigate sheaves which achieve the bound above, and find that they are all globally generated vector bundles given by a specific type of resolution. Such bundles are said to be \emph{balanced twisted Steiner-like} (see Definition \ref{def:Steiner-like_sheaf}), motivated by the well-studied notion of Steiner bundles on $\PP^2$ (see \cite{Brambilla05,Brambilla08,Huizenga12}). The following result is similar to the case of $\PP^2$ (see \cite{CHR25}).

\begin{theorem}\label{SummaryThm02}
	(Theorem \ref{thm:full_structure_thm_maximal_sheaves_with_vanishing_h1}.) Let $E$ be a $\mu_H$-semistable sheaf on $\Quadric$ of rank $r\geq1$ and slope $\mu\geq0$ such that $h^0(E)=\beta_{r,\mu}$. Then $h^1(E)=h^2(E)=0$ and $c_1(E)=r(\alpha-1,\alpha-1)+(m,n)$ for some non-negative integers $m,n$, with $\alpha:=\alpha_{\mu}$ and $m+n<2r$. Moreover, $E$ is a globally generated vector bundle having a resolution $$0\to \OO(\alpha-2,\alpha-1)^m\oplus\OO(\alpha-1,\alpha-2)^n\to \OO(\alpha-1,\alpha-1)^{r+m+n}\to  E\to 0.$$
\end{theorem}

The cohomology vanishing result above shows that certain Brill-Noether loci are either the entire moduli space, or empty. This is summarized below.

\begin{corollary}
	(Corollary \ref{cor:BN_loci}.) Suppose $M:=M_{\Quadric,H}(r,c_1,\chi)$ is a nonempty moduli space of $H$-Gieseker semistable sheaves on $\Quadric$ of rank $r\geq1$, first Chern class $c_1$ and Euler characteristic $\chi$, with associated slope $\mu\geq0$. Let $\beta:=\beta_{r,\mu}$. Then the Brill-Noether locus $B^{\beta}\subseteq M$ is nonempty if and only if $\chi=\beta$. In the case that $\chi=\beta$, the locus $B^k$ is all of $M$ if and only if $k\leq\beta$, and is empty otherwise.
\end{corollary}

\subsubsection{Bound for the non-globally generated case}

If the sheaf $E$ does not have good global generation properties, the bound in Theorem \ref{SummaryThm01} can be improved.

\begin{theorem}\label{SummaryThm03}
	(Theorem \ref{thm:stratified_bound}.) Let $E$ be a torsion-free sheaf of rank $r\geq1$ on $\Quadric$ with $h^0(E)>0$. Let $S\subset E$ be the image of the canonical evaluation map $\OO_{\Quadric}\otimes H^0(E)\to E$. Let $\mu'\in\left[0,\mumax(E)\right]$ be the largest rational number that can be written with denominator belonging to the set $\{2,4,\ldots,2\rk(S)\}$. Then $$h^0(E)\leq\beta_{\rk(S),\mu'}.$$
\end{theorem}

We find that $\mu_H$-semistable sheaves of small slope that achieve the above bound are extensions of sheaves without global sections by balanced twisted Steiner-like bundles.

\begin{theorem}
	(Theorem \ref{thm:non_gen_gg_maximal_small_slope_more_general}.) Let $E$ be a $\mu_H$-semistable sheaf of rank $r\geq1$ on $\Quadric$ with $\mu:=\mu_H(E)\in[0,1)$ and $h^0(E)>0$. Let $S\subset E$ be the image of the canonical evaluation map $\OO_{\Quadric}\otimes H^0(E)\to E$, and let $\mu'\in\left[0,\mu\right]$ be the largest rational number that can be written with denominator belonging to the set $\{2,4,\ldots,2\rk(S)\}$. Further suppose that $$h^0(E)=\beta_{\rk(S),\mu'}.$$ Then $S$ is a balanced twisted Steiner-like bundle of slope $\mu'$, and $E$ fits in an exact sequence of the form $$0\to S\to E\to Q\to 0$$ with $h^0(Q)=0$.
\end{theorem}

We illustrate the sharpness of the bound furnished by Theorem \ref{SummaryThm03} in Proposition \ref{prop:non_gen_gg_maximal_sharpness}. For this purpose, we prove a result about the semistability of extensions on del Pezzo surfaces. We also hope this result will be useful elsewhere.

\begin{theorem}
	(Theorem \ref{thm-ss_ext_elliptic}.) Let $X$ be a del Pezzo surface with anticanonical polarization, and let $C\in|-K_X|$ be a smooth elliptic curve with an arbitrary polarization. Suppose $M'=M_{X,-K_X}(r',c_1',\chi')$ and $M''=M_{X,-K_X}(r'',c_1'',\chi'')$ are moduli spaces of semistable sheaves on $X$, with $\chi''\ll0$. Let $A$ and $B$ be general elements of $M'$ and $M''$ respectively. Assume that $A$ is stable on $X$, that $$\mu(A|_C)<\mu(B|_C),$$ and that $\gcd(r',c_1'\!\cdot\!C)=\gcd(r'',c_1''\!\cdot\!C)=1$. Then for a general extension $$0\to A\to V\to B\to 0,$$ $V$ is $\mu_{-K_X}$-semistable. If furthermore $\gcd(r'+r'',(c_1'+c_1'')\!\cdot\!C)=1$, then $V$ is $\mu_{-K_X}$-stable.
\end{theorem}

\subsubsection{Bound for the unbalanced case}

Theorem \ref{SummaryThm02} shows that sheaves achieving the bound in Theorem \ref{SummaryThm01} are \emph{balanced} in the sense that the components of its first Chern class are close to each other. This leads us to improve the bound for sheaves that are not so balanced, but still have some good global generation properties.

\begin{theorem}
	(Theorem \ref{thm:better_bound_sections_stratified}.) Let $E$ be a torsion-free sheaf of rank $r\geq1$ on $\Quadric$ with $c_1(E)=(a,b)$. Without loss of generality, assume $b\geq a$. Suppose $E$ is generically globally generated, and $j$ is the largest integer with $$0\leq j\leq \Bigg\lfloor\frac{b-a}{r}\Bigg\rfloor$$ for which $E\!\left(0,-j\right)$ is generically globally generated. Then $$h^0(E)\leq \min\!\left\{\beta_{r,\mu''+1/2}+j(r+a),\,\, \beta_{r,\mu''}+(j+1)(r+a)\right\},$$ where $$\mu'':=\mumax(E)-\frac{(j+1)}{2}.$$
\end{theorem}

This bound is sharp for \emph{all} line bundles $\OO_{\Quadric}(a,b)$ with $a,b\geq0$.

In Theorems \ref{thm:strongly_maximal_structure} and \ref{thm:strongly_maximal_structure_2}, we show that sheaves achieving this bound are often just twisted Steiner-like bundles.

\subsection*{Organization of the paper} In \S\ref{sec-prelim}, we recall some basic facts about moduli spaces of semistable sheaves on polarized varieties. In \S\ref{sec:bounds_quadric}, we prove Theorems \ref{thm:bound_sections}, \ref{thm:better_bound_sections_stratified} and \ref{thm:stratified_bound}, which provide upper bounds on $h^0(E)$ on $\Quadric$ based on the balancedness of $c_1(E)$ and global generation properties of $E$. In \S\ref{sec:maximal_balanced}, we have Theorem \ref{thm:full_structure_thm_maximal_sheaves_with_vanishing_h1} which classifies sheaves attaining the bound for $h^0$ in the balanced case. In \S\ref{sec:maximal_unbalanced}, our Theorems \ref{thm:strongly_maximal_structure} and \ref{thm:strongly_maximal_structure_2} discuss sheaves attaining the bound for $h^0$ in the unbalanced case. In \S\ref{sec:semistab_extn}, we prove the semistability of extensions on elliptic curves and del Pezzo surfaces. Finally, in \S\ref{sec:maximal_non_gg}, Theorem \ref{thm:non_gen_gg_maximal_small_slope_more_general} and Proposition \ref{prop:non_gen_gg_maximal_sharpness} discuss sheaves that achieve the bound in the non-globally generated case.

\subsection*{Acknowledgments} We would like to thank Jack Huizenga for \emph{uncountably many} discussions on this project. We also thank Izzet Coskun and Jack for allowing the author to include their arguments regarding semistability of extensions in section \ref{sec:semistab_extn}.

\section{Preliminaries}\label{sec-prelim}

In this section, we recall some facts and definitions about stability, and moduli spaces of sheaves on $\Quadric$.

\subsection{Global generation} 
For a torsion-free coherent sheaf $E$ on a smooth projective variety $X$, let $$\ev: \OO_X\otimes H^0(E)\to E$$ be the canonical evaluation morphism. We shall use the following associated terminology.
\begin{enumerate}
	\item $E$ is said to be \emph{globally generated} if $\ev$ is surjective. 
	\item $E$ is said to be \emph{generically globally generated} if the cokernel of $\ev$ is torsion.
	\item $E$ is said to be \emph{globally generated in codimension 1} if the cokernel of $\ev$ is supported on a subscheme of dimension at most $0$.
\end{enumerate}
Then we have the following straightforward implications:
$$\mbox{global generation} \Rightarrow \mbox{global generation in codimension 1} \Rightarrow \mbox{generic global generation.}$$

We also note that if $C$ is a smooth rational curve on $X$, and $E$ is locally free along $C$ and generically globally generated along $C$, i.e., $E|_C$ is a generically globally generated vector bundle on $C\cong\PP^1$, then $E|_C$ splits as a direct sum $$E|_C\cong\bigoplus_i \OO_C(d_i)$$ for some non-negative integers $d_i$. It follows that $h^1(E|_C)=0$ and $$h^0(E|_C)=\chi(E|_C)=\rk(E)+c_1(E)\cdot C$$ by Riemann-Roch.

\subsection{Slope stability}

Given a torsion-free coherent sheaf $E$ on a polarized projective variety $(X,L)$, the \emph{slope} $\mu_L(E)$ of $E$ is defined as 
$$\mu_L(E) = \frac{c_1(E)\cdot L^{n-1} }{\rk(E)L^n},$$ where $n$ denotes the dimension of $X$. We say that $E$ is \emph{$\mu_L$-(semi)stable} (or \emph{slope (semi)stable}) if for every proper subsheaf $F$ of $E$, we have $$\mu_L(F)\leqor\mu_L(E).$$
Every torsion-free sheaf $E$ admits a unique \emph{Harder-Narasimhan filtration}
$$0=E_0\subset E_1\subset\cdots\subset E_{\ell}=E$$ such that each $\gr_i := E_i/E_{i-1}$ is $\mu_L$-semistable and $$\mu(\gr_1)>\cdots>\mu(\gr_{\ell}).$$ We define $\mumax(E) := \mu(\gr_1)$ and $\mumin(E) := \mu(\gr_{\ell})$.

An important consequence of $\mu_L$-semistability is that if $E$ and $F$ are $\mu_L$-semistable sheaves with $\mu_L(E)>\mu_L(F)$, then $\Hom(E,F)=0$. In particular, if $X$ is $\Quadric$, $L$ is $\OO_{\Quadric}(1,1)$ and $E$ is a $\mu_L$-semistable sheaf on $\Quadric$ with slope $\mu_L(E)>-2$, then by Serre duality we get $$H^2(E)=\Ext^2(\OO,E)\cong\Hom(E,\OO(-2,-2))^*=0$$ since $\mu_L(\OO(-2,-2))=-2$.

\subsection{Gieseker stability}

For a pure $n$-dimensional sheaf $E$ on a polarized projective variety $(X,L)$, the \emph{Hilbert polynomial} $P_E(m)$ and the \emph{reduced Hilbert polynomial} $p_E(m)$ of $E$ are defined as $$P_E(m) := \chi(E(mL)) = a_n\frac{m^n}{n!} + \mbox{lower order terms}, \mbox{ and } p_E(m) := \frac{P_E(m)}{a_n}.$$
We say that $E$ is \emph{$L$-Gieseker (semi)stable} if for every proper subsheaf $F$ of $E$, we have $$p_F(m)\leqor p_E(m)$$ for $m\gg0$. Without any adjective, the term (semi)stability usually means Gieseker (semi)stability with respect to the polarization in question. For torsion-free sheaves on $X$, we have the implications
$$\mu_L\mbox{-stable} \Rightarrow \mbox{stable} \Rightarrow \mbox{semistable} \Rightarrow \mu_L\mbox{-semistable}.$$ 
Every pure coherent sheaf admits a unique Harder-Narasimhan filtration for Gieseker semistability. Moreover, every semistable sheaf admits a \emph{Jordan-H\"{o}lder filtration} with stable quotients, which need not be unique. However, the associated graded object is unique up to isomorphism. Two semistable sheaves are said to be \emph{S-equivalent} if their associated graded objects for the Jordan-H\"{o}lder filtration are isomorphic.

\subsection{Moduli spaces} Given a polarized projective variety $(X,L)$, there are projective moduli spaces of sheaves $M_{X,L}({\bf v})$ parameterizing $S$-equivalence classes of semistable sheaves on $X$ with a fixed Chern character ${\bf v}$ (see \cite{Gieseker,Maruyama}). The subscript in $M_{X,L}({\bf v})$ will be dropped if $X$ and $L$ are clear from the context.

Several authors have studied these moduli spaces for $X=\Quadric$ (see \cite{CH-Semistable_Hirzebruch,Rudakov-Semistab,Pedchenko}). We will write a Chern character $\bf v$ as $(r,c_1,\chi)$, where $r$ is the rank, and $\chi$ is the Euler characteristic. When nonempty, the moduli space $M:=M_{\Quadric,\OO(1,1)}(r, c_1, \chi)$ is a normal, factorial,  irreducible projective variety that is smooth along the locus $M^s$ of stable sheaves.

We let $B^k({\bf v}) \subset M_{X,L}({\bf v})$ denote the \emph{Brill-Noether locus} defined as the closure of the locus of stable sheaves $E$ on $X$ for which $h^0(E)\geq k$. The Brill-Noether loci have a natural determinantal scheme structure in the stable locus $M_{X,L}({\bf v})^s$ (see \cite{CostaMR}).

\section{The upper bounds}\label{sec:bounds_quadric}

In this section, we prove upper bounds on $h^0(E)$ for torsion-free sheaves $E$ on $\Quadric$.

For notational simplicity, we shall write $\OO_{\Quadric}(a,b)$ as $(a,b)$, for any integers $a,b$. Let $H$ denote the ample class $(1,1)$. For any torsion-free sheaf $E$ of rank $r\geq1$ on $\Quadric$ with $c_1(E)=(a,b)$, its \emph{slope} is $$\mu_H(E)=\frac{c_1(E)\cdot H}{rH^2}=\frac{a+b}{2r}.$$

We begin with an easy lemma.

\begin{lemma}\label{lem:ggg_implies_a,b_nonnegative}
	Let $E$ be a generically globally generated, torsion-free sheaf on $\Quadric$ of rank $r\geq1$ and $c_1(E)=(a,b)$. Then $a\geq0$ and $b\geq0$. In particular, $$\mumax(E)\geq\mu(E)=\frac{a+b}{2r}\geq0.$$
\end{lemma}

\begin{proof}
	Let $L\subset\Quadric$ be a general curve of class $(0,1)$, so that $L$ does not meet any singularities of $E$ and $E$ is globally generated at a general point of $L$. Then $E|_L$ is a generically globally generated vector bundle on $L\cong \PP^1$, and it follows that $E|_L$ splits as a direct sum of line bundles $E|_L \cong \bigoplus_i  \OO_L(d_i)$ with $d_i\geq 0$ for each $i$. Then $a=c_1(E)\cdot L=\deg(E|_L)=\sum_i d_i\geq0$. Similarly, using a general curve $C$ of class $(1,0)$ we get $b\geq0$.
\end{proof}

We now introduce some notation that is going to be very useful throughout this paper.

\begin{definition}\label{def:alpha_beta}
	Let $r\geq1$ be an integer, and let $\mu\geq-1$ be a rational number. Then we define $$\alpha_{\mu}:=\lfloor \mu\rfloor + 1\quad\mbox{ and }\quad\beta_{r,\mu}:=r\alpha_{\mu}\left(2\mu-\alpha_{\mu}+2\right).$$
\end{definition}

Since $\mu\geq\alpha_{\mu}-1$, we note that $\beta_{r,\mu}\geq r\alpha_{\mu}^2\geq0$.

\subsection{General bound}

We can now state and prove the first bound on $h^0$, which turns out to be good when the components of the first Chern class are close to each other.

\begin{theorem}\label{thm:bound_sections}
	Let $E$ be a torsion-free sheaf of rank $r\geq1$ on $\Quadric$ with $\mu_{\max}(E)\geq-1$. Then $$h^0(E)\leq\beta_{r,\mu_{\max}(E)}.$$
\end{theorem}

\begin{proof}
	For notational simplicity, let $\mu:=\mumax(E)$ and $\alpha:=\alpha_{\mu}$. We shall induct on both $\alpha$ and $r$.
	
	For the base case, if $\alpha=0$, then $\mu<0$. This means that $$h^0(E)\leq\sum_i h^0(\gr_i)=0=\beta_{r,\mu},$$ where the $\gr_i$'s are the Harder-Narasimhan factors of $E$, and are therefore $\mu_H$-semistable. Next, if $r=1$, then $E\cong I_Z(a,b)$ for some integers $a,b$ and some $0$-dimensional closed subscheme $Z$ of $\Quadric$. So $$h^0(E)\leq h^0(\OO(a,b))=\begin{cases}
		0&\mbox{ if }a<0\mbox{ or }b<0\\
		(a+1)(b+1)&\mbox{ if }a\geq0\mbox{ and }b\geq0,
	\end{cases}$$ while $$0\leq\beta_{1,\mu}=\alpha(2\mu-\alpha+2)=\begin{cases}
	\left(\dfrac{a+b}{2}+1\right)^{\!2}&\mbox{ if }a+b\mbox{ is even}\\[15pt]
	\left(\dfrac{a+b}{2}+\dfrac{1}{2}\right)\!\left(\dfrac{a+b}{2}+\dfrac{3}{2}\right)&\mbox{ if }a+b\mbox{ is odd}.
\end{cases}$$ It is now straightforward to see that $h^0(\OO(a,b))\leq\beta_{1,\mu}$.
	
	We now assume $\alpha>0$ and $r>1$. Then $\beta_{r,\mu}>0$, and we assume that $E$ has a non-zero global section, as otherwise we are done.
	
	There are two cases to consider based on whether $E$ is generically globally generated.
	
	\emph{Case 1:} Suppose $E$ is generically globally generated. Choose a general curve $C$ of class $H=(1,1)$, so that $C$ is a smooth irreducible rational curve avoiding any singularities of $E$, and $E$ is globally generated at a general point of $C$. Then $E|_C$ is a generically globally generated vector bundle on $C\cong\PP^1$, so $E|_C$ splits as a direct sum of line bundles $E|_C \cong \bigoplus_i  \OO_C(d_i)$ with each $d_i\geq 0$. Write $c_1(E)=(a,b)$. Since $C$ is of class $H$, $a+b = c_1(E)\cdot H=\deg(E|_C) = \sum_i d_i$.  This gives $h^0(E|_C) = \sum_i (d_i+1) = r+a+b=r+c_1$ where $c_1$ denotes $c_1(E)\cdot H$. The restriction exact sequence $$0\to E(-1,-1)\to E\to E|_C\to 0$$ then yields \begin{align*}
		h^0(E)&\leq h^0(E(-1,-1))+h^0(E|_C)\\[5pt]
		&\leq\beta_{r,\mu-1} + r+c_1\qquad\quad\quad\qquad\quad\,\,\,\,\left(\mbox{since }\mumax(E(-1,-1))=\mu-1\right)\\[5pt]
		&=\beta_{r,\mu}-2r\mu+c_1\\[5pt]
		&\leq \beta_{r,\mu}
	\end{align*} by induction on $\alpha$, and the fact that $$\frac{c_1}{2r}=\mu(E)\leq\mumax(E)=\mu.$$
	
	\emph{Case 2:} Suppose $E$ is not generically globally generated. Let $S\subset E$ be the image of the canonical evaluation map $\OO_{\Quadric}\otimes H^0(E)\to E$, and let $r'$ be its rank. Let $\mu':=\mumax(S)$ and $\alpha':=\lfloor\mu'\rfloor+1$. Note that $\mu'\leq\mu$. Thus, $\alpha'\leq\alpha$. Since $S$ is globally generated, $\mu'\geq0>-1$ by Lemma \ref{lem:ggg_implies_a,b_nonnegative} above. Because $E$ is not generically globally generated, $r'<r$. So the asserted bound holds for $S$ by induction on $r$.
	
	If $\alpha'=\alpha$, then \begin{align*}
		h^0(E)&=h^0(S)\\
		&\leq\beta_{r',\mu'}\\
		&=r'\alpha\left(2\mu'-\alpha+2\right)\\
		&<r\alpha\left(2\mu-\alpha+2\right)\\
		&=\beta_{r,\mu}.
	\end{align*}
	
	On the other hand, suppose $\alpha'<\alpha$. Then we have \begin{align*}
		h^0(E)&=h^0(S)\\
		&\leq\beta_{r',\mu'}\\
		&=r'\alpha'\left(2\mu'-\alpha'+2\right)\\
		&<r'\alpha'\left(2\alpha'-\alpha'+2\right)\qquad\qquad\qquad\,\,\,\,\,\,(\mbox{since }\mu'<\alpha')\\
		&=r'\alpha'\left(\alpha'+2\right)\\
		&\leq(r-1)(\alpha-1)\left(\alpha+1\right)\qquad\quad\qquad\,\,\,\,(\mbox{since }r'\leq r-1\mbox{ and }\alpha'\leq\alpha-1)\\
		&=(r-1)(\alpha^2-1)\\
		&<r\alpha^2\\
		&=r\alpha\left(2(\alpha-1)-\alpha+2\right)\\
		&\leq\beta_{r,\mu}\qquad\qquad\qquad\qquad\qquad\qquad\,\,\quad\,\,(\mbox{since }\alpha-1\leq\mu).
	\end{align*}
\end{proof}

We can improve the above bound in case $E$ is fairly close to being $\mu_H$-semistable.

\begin{corollary}\label{cor:strengthened_bound}
	Let $E$ be a torsion-free sheaf of rank $r\geq1$ on $\Quadric$ with $\mu:=\mu(E)\geq-1$ and $$\alpha-1\leq\mumin(E)\leq\mumax(E)<\alpha,$$ where $\alpha:=\alpha_{\mu}$. Then $$h^0(E)\leq\beta_{r,\mu}.$$
\end{corollary}

\begin{proof}
	Let $\gr_1,\ldots,\gr_{\ell}$ be the Harder-Narasimhan factors of $E$. Let $r_i:=\rk(\gr_i)$ and $\mu_i:=\mu(\gr_i)$ for each $i$. Note that each $\alpha_{\mu_i}$ equals $\alpha$. Since each $\gr_i$ is $\mu_H$-semistable and $\mu_i\geq\mumin(E)\geq\alpha-1\geq-1$, Theorem \ref{thm:bound_sections} implies that $$h^0(\gr_i)\leq\beta_{r_i,\mu_i}=r_i\alpha\left(2\mu_i-\alpha+2\right)$$  for each $i$. It follows that \begin{align*}
		h^0(E)&\leq\sum_i h^0(\gr_i)\\
		&\leq\sum_i r_i\alpha\left(2\mu_i-\alpha+2\right)\\
		&=r\alpha\left(2\mu-\alpha+2\right)\\
		&=\beta_{r,\mu},
	\end{align*} since $r=\sum_i r_i$ and $$r\mu=\frac{c_1(E)\cdot H}{2}=\sum_i\frac{c_1(\gr_i)\cdot H}{2}=\sum_i r_i\mu_i.$$
\end{proof}

The bound in Theorem \ref{thm:bound_sections} works well for \emph{balanced} sheaves (i.e., sheaves whose first Chern class' components are close to each other), as illustrated by the following remark.

\begin{remark}\label{rmk:maximality_of_balanced_direct_sum_bundles}
	Note that $$\mumax\!\left(\OO(m,m)^{\oplus r}\right)=\mu_H\!\left(\OO(m,m)^{\oplus r}\right)=m$$ and $$\beta_{r,m}=r(m+1)^2=h^0\!\left(\OO(m,m)^{\oplus r}\right)$$ for integers $m\geq-1$ and $r\geq1$. Also, if $m,n\geq-1$ are integers with $|m-n|=1$, we see that $$\mu:=\mumax\!\left(\OO(m,n)^{\oplus r}\right)=\mu_H\!\left(\OO(m,n)^{\oplus r}\right)=(m+n)/2$$ and $$\beta_{r,\mu}=r(m+1)(n+1)=h^0\!\left(\OO(m,n)^{\oplus r}\right)$$ since $\lfloor\mu\rfloor=\min\{m,n\}$. Thus, copies of balanced line bundles achieve the bound in Theorem \ref{thm:bound_sections}. On the other hand, examples like $\OO(0,2)$ show that if the gap between $m$ and $n$ is larger than $1$, then $h^0(\OO(m,n))$ need not achieve the bound. This leads us to consider improving the bound from Theorem \ref{thm:bound_sections} for the case of unbalanced sheaves.
\end{remark}

\subsection{Bound for the unbalanced case}

In this subsection, we prove a better bound in the case that $c_1(E)=(a,b)$ with $a$ and $b$ not being very close to each other. The main idea is that we can twist down one of the factors till we get closer to being balanced.

\begin{theorem}\label{thm:better_bound_sections_stratified}
	Let $E$ be a torsion-free sheaf of rank $r\geq1$ on $\Quadric$ with $c_1(E)=(a,b)$. Without loss of generality, assume $b\geq a$. Suppose $E$ is generically globally generated, and $j$ is the largest integer with $$0\leq j\leq \Bigg\lfloor\frac{b-a}{r}\Bigg\rfloor$$ for which $E\!\left(0,-j\right)$ is generically globally generated. Then $$h^0(E)\leq \min\!\left\{\beta_{r,\mu''+1/2}+j(r+a),\,\, \beta_{r,\mu''}+(j+1)(r+a)\right\},$$ where $$\mu'':=\mumax(E)-\frac{(j+1)}{2}.$$
\end{theorem}

\begin{proof}
	Let $L$ be a general line of class $(0,1)$ such that $$E\!\left(0,-j\right)$$ is globally generated at a general point of $L$, and none of its singularities are on $L$. Then each of $$E\!\left(0,-j+1\right),E\!\left(0,-j+2\right),\ldots,E$$ also has the same properties with respect to $L$. Moreover, for each of these, the restriction to $L$ has $h^0$ equal to $r+a$ by Riemann-Roch. Also note that $$\mumax\!\left(E\!\left(0,-j\right)\right)=\mumax(E)-\frac{j}{2}=\mu''+\frac{1}{2}$$ and $$\mumax\!\left(E\!\left(0,-j-1\right)\right)=\mumax(E)-\frac{(j+1)}{2}=\mu''\geq-1,$$ where the last inequality is due to Lemma \ref{lem:ggg_implies_a,b_nonnegative} being applied on $E\!\left(0,-j\right)$. Thus, Theorem \ref{thm:bound_sections} yields the inequalities $$h^0\!\left(E\!\left(0,-j\right)\right)\leq\beta_{r,\mu''+1/2}$$ and $$h^0\!\left(E\!\left(0,-j-1\right)\right)\leq\beta_{r,\mu''}.$$ Now, using twists of the restriction sequence $$0\to E(0,-1)\to E\to E|_L\to 0$$ repeatedly, we get \begin{align*}
		h^0(E)&\leq h^0(E(0,-1))+r+a\\
		&\leq h^0(E(0,-2))+2(r+a)\\
		&\leq\cdots\\
		&\leq h^0\!\left(E\!\left(0,-j\right)\right)+j(r+a)\\
		&\leq h^0\!\left(E\!\left(0,-j-1\right)\right)+(j+1)(r+a).\end{align*} The assertion follows.
\end{proof}

\begin{remark}
	Unlike the bound from Theorem \ref{thm:bound_sections}, the bound from Theorem \ref{thm:better_bound_sections_stratified} is actually sharp for copies $\OO(m,n)^{\oplus r}$ of all line bundles with $m,n\geq0$ (cf. Remark \ref{rmk:maximality_of_balanced_direct_sum_bundles}). To see this, let $0\leq m\leq n$ be integers and $E$ be the rank $r$ $\mu_H$-semistable vector bundle $\OO(m,n)^{\oplus r}$. Since $$E\!\left(0,-\Bigg\lfloor\frac{rn-rm}{r}\Bigg\rfloor\right)=\OO(m,m)^{\oplus r}$$ is globally generated, we have $$j=\Bigg\lfloor\frac{rn-rm}{r}\Bigg\rfloor=n-m$$ and $$\mu''=\frac{rm+rn}{2r}-\frac{(n-m+1)}{2}=m-\frac{1}{2}.$$ So $$\beta_{r,\mu''+1/2}+j(r+rm)=r(m+1)(2m-m-1+2)+(n-m)(r+rm)=r(m+1)(n+1)$$ and $$\beta_{r,\mu''}+(j+1)(r+rm)=rm(2m-1-m+2)+(n-m+1)(r+rm)=r(m+1)(n+1).$$ Thus, the bound furnished by the above result exactly equals $h^0(E)$.
\end{remark}

We now observe that usually the more we can twist, the better the bound we get from Theorem \ref{thm:better_bound_sections_stratified} above. We introduce the following notation. Given non-negative integers $a\leq b$ and a positive integer $r$, let $\mu:=\frac{a+b}{2r}$. For any integer $$0\leq\ell\leq \Bigg\lfloor\frac{b-a}{r}\Bigg\rfloor,$$ we define $$\mu''_{\ell}:=\mu-\frac{(\ell+1)}{2}$$ and $$\theta_{\ell}:=\beta_{r,\mu''_{\ell}}+(\ell+1)(r+a).$$

Suppose $E$ is a generically globally generated $\mu_H$-semistable sheaf of rank $r$ on $\Quadric$ with $c_1(E)=(a,b)$, and let $j$ be the largest integer with $$0\leq j\leq \Bigg\lfloor\frac{b-a}{r}\Bigg\rfloor$$ for which $E\!\left(0,-j\right)$ is generically globally generated. Then Theorem \ref{thm:better_bound_sections_stratified} says that $$h^0(E)\leq\min\left\{\theta_{j-1},\,\, \theta_j\right\}.$$ The following numerical result shows that in Theorem \ref{thm:better_bound_sections_stratified}, the higher $j$ is, the smaller the upper bound obtained.

\begin{proposition}\label{prop:monotonicity_of_theta}
	Fix non-negative integers $a\leq b$ and a positive integer $r$. Let $\mu:=\dfrac{a+b}{2r}$. Let $$0\leq\ell\leq \Bigg\lfloor\frac{b-a}{r}\Bigg\rfloor$$ be an integer. Then $\theta_{\ell}<\theta_{\ell-1}$, unless $$\ell=\Bigg\lfloor\frac{b-a}{r}\Bigg\rfloor,$$ or $$\ell=\Bigg\lfloor\frac{b-a}{r}\Bigg\rfloor-1\quad\mbox{ and }\quad\lfloor\mu''_{\ell-1}\rfloor=\lfloor\mu''_{\ell}\rfloor.$$
\end{proposition}

\begin{proof}
	Let $\alpha''_{\ell}:=\lfloor\mu''_{\ell}\rfloor+1$ and $\alpha''_{\ell-1}:=\lfloor\mu''_{\ell-1}\rfloor+1$. Since $$\mu''_{\ell}=\mu''_{\ell-1}-\frac{1}{2},$$ either $\alpha''_{\ell}=\alpha''_{\ell-1}$ or $\alpha''_{\ell}=\alpha''_{\ell-1}-1$.
	
	First suppose $\alpha''_{\ell}=\alpha''_{\ell-1}$. Then $\theta_{\ell}-\theta_{\ell-1}=r+a-r\alpha''_{\ell-1}$. Suppose $r+a-r\alpha''_{\ell-1}\geq0$. This is equivalent to having $$\frac{r+a}{r}-1\geq\lfloor\mu''_{\ell}\rfloor,$$ since $\alpha''_{\ell}=\alpha''_{\ell-1}$. Suppose $$\frac{r+a}{r}-1>\mu''_{\ell}.$$ Then $$2r+2a-2r>a+b-r(\ell+1),$$ which implies $$\frac{b-a}{r}<\ell+1\leq\Bigg\lfloor\frac{b-a}{r}\Bigg\rfloor+1.$$ This is absurd, unless $$\ell=\Bigg\lfloor\frac{b-a}{r}\Bigg\rfloor.$$ So, assuming $$\ell\leq\Bigg\lfloor\frac{b-a}{r}\Bigg\rfloor-1,$$ we have  $$\frac{r+a}{r}-1\leq\mu''_{\ell}.$$ Since $$\lfloor\mu''_{\ell}\rfloor=\lfloor\mu''_{\ell-1}\rfloor\leq\frac{r+a}{r}-1,$$ we must have $$\mu''_{\ell}-\left(\frac{r+a}{r}-1\right)<\frac{1}{2}.$$ So $$\frac{a+b}{2r}-\frac{(\ell+1)}{2}-\frac{(r+a)}{r}+1<\frac{1}{2},$$ which implies $$\ell>\frac{b-a}{r}-2,$$ whence $$\ell\geq\Bigg\lfloor\frac{b-a}{r}\Bigg\rfloor-1.$$ It follows that $$\ell=\Bigg\lfloor\frac{b-a}{r}\Bigg\rfloor-1.$$
	
	Next, suppose $\alpha''_{\ell}=\alpha''_{\ell-1}-1$. Then $\theta_{\ell}-\theta_{\ell-1}=r+a-r\left(2\mu''_{\ell-1}-\alpha''_{\ell-1}+2\right)$. Suppose $$r+a-r\left(2\mu''_{\ell-1}-\alpha''_{\ell-1}+2\right)\geq0.$$ Then $$r+a-r\left(2\left(\frac{a+b}{2r}-\frac{\ell}{2}\right)-\alpha''_{\ell-1}+2\right)\geq0.$$ So $$1+\frac{a}{r}\geq\frac{a+b}{r}-\ell-\alpha''_{\ell-1}+2,$$ which implies $$\frac{b}{r}\leq \ell-1+\alpha''_{\ell-1}\leq \ell-1+\mu''_{\ell-1}+1=\ell+\frac{a+b}{2r}-\frac{\ell}{2}.$$ Thus, $$\frac{b-a}{2r}\leq\frac{\ell}{2},$$ whence $$\ell\geq\frac{b-a}{r}\geq\Bigg\lfloor\frac{b-a}{r}\Bigg\rfloor.$$ Thus, we must have $$\ell=\Bigg\lfloor\frac{b-a}{r}\Bigg\rfloor.$$
\end{proof}

\subsection{Bound for the non-globally generated case}

We now refine the bound on $h^0(E)$ furnished by Theorem \ref{thm:bound_sections} for  sheaves $E$ on $\Quadric$ that are not generically globally generated. We begin with a preliminary lemma.

\begin{lemma}\label{lem:monotonicity_of_beta_1}
	The function $\beta_{r,\mu}$ defined in Definition \ref{def:alpha_beta} satisfies the following monotonicity properties. \begin{enumerate}
		\item For any rational number $\mu\geq-1$ and any integers $r\geq r'\geq1$, we have $\beta_{r',\mu}\leq\beta_{r,\mu}$.
		\item For any rational number $\mu\geq0$ and any integers $r>r'\geq1$, we have $\beta_{r',\mu}<\beta_{r,\mu}$.
		\item For any integer $r\geq1$ and any rational numbers $\mu\geq\mu'\geq-1$, we have $\beta_{r,\mu'}\leq\beta_{r,\mu}$.
		\item For any integer $r\geq1$ and any rational numbers $\mu>\mu'\geq-1$ with $\mu\geq0$, we have $\beta_{r,\mu'}<\beta_{r,\mu}$.
	\end{enumerate}
\end{lemma}

\begin{proof}
	The proofs of (i) and (ii) are clear, since $\alpha_{\mu}(2\mu-\alpha_{\mu}+2)\geq0$ with strict inequality in the situation of (ii) (because $\mu\geq\alpha_{\mu}-1$).
	
	We prove (iii) now. Let $\alpha:=\alpha_{\mu}$ and $\alpha':=\alpha_{\mu'}$.
	
	\emph{Case 1:} $\alpha'=\alpha$. Then $\beta_{r,\mu'}=r\alpha(2\mu'-\alpha+2)\leq r\alpha(2\mu-\alpha+2)=\beta_{r,\mu}$.
	
	\emph{Case 2:} $\alpha'<\alpha$. We have \begin{align*}
		\beta_{r,\mu'}&=r\alpha'(2\mu'-\alpha'+2)\\
		&\leq r(\alpha-1)(2\mu'-\alpha'+2)\qquad\qquad\qquad\,(\mbox{since }\alpha'\leq\alpha-1)\\
		&\leq r(\alpha-1)(\alpha'+2)\qquad\qquad\qquad\qquad\,\,\,\,\,(\mbox{since }\mu'<\alpha')\\
		&\leq r(\alpha-1)(\alpha+1)\qquad\qquad\qquad\qquad\quad\,(\mbox{since }\alpha'\leq\alpha-1)\\
		&=r(\alpha^2-1)\\
		&<r\alpha^2\\
		&=r\alpha(2(\alpha-1)-\alpha+2)\\
		&\leq\beta_{r,\mu}\qquad\qquad\qquad\qquad\qquad\qquad\,\,\qquad(\mbox{since }\alpha-1\leq\mu).
	\end{align*}

	Assertion (iv) is also clear now.
\end{proof}

Based on how globally generated the sheaf $E$ is, one can strengthen the upper bound on $h^0(E)$ from Theorem \ref{thm:bound_sections}.

\begin{theorem}\label{thm:stratified_bound}
	Let $E$ be a torsion-free sheaf of rank $r\geq1$ on $\Quadric$ with $h^0(E)>0$. Let $S\subset E$ be the image of the canonical evaluation map $\OO_{\Quadric}\otimes H^0(E)\to E$. Let $\mu'\in\left[0,\mumax(E)\right]$ be the largest rational number that can be written with denominator belonging to the set $\{2,4,\ldots,2\rk(S)\}$. Then $$h^0(E)\leq\beta_{\rk(S),\mu'}.$$
\end{theorem}

\begin{proof}
	Since $S$ is a globally generated subsheaf of $E$, we have $0\leq\mumax(S)\leq\mumax(E)$. Applying Theorem \ref{thm:bound_sections} to the torsion-free sheaf $S$, we get $$h^0(E)=h^0(S)\leq\beta_{\rk(S),\mumax(S)}.$$ Since $\mumax(S)$ is the slope of a subsheaf of $S$, it is a rational number having denominator in the set $\{2,4,\ldots,2\rk(S)\}$. By Lemma \ref{lem:monotonicity_of_beta_1}, it follows that $\beta_{\rk(S),\mumax(S)}\leq\beta_{\rk(S),\mu'}$. The result follows.
\end{proof}

As a corollary, we get the promised bound for non-globally generated sheaves.

\begin{corollary}\label{cor:bound_sections_non_gen_gg}
	Let $E$ be a torsion-free sheaf of rank $r\geq2$ on $\Quadric$ with $\mumax(E)\geq-1$. Suppose $E$ is not generically globally generated. Let $\mu'\in\left[-1,\mumax(E)\right]$ be the largest rational number that can be written with denominator belonging to the set $\{2,4,\ldots,2(r-1)\}$. Then $$h^0(E)\leq\beta_{r-1,\mu'}.$$
\end{corollary}

\begin{proof}
	If $h^0(E)=0$, there is nothing to prove since $\beta_{r-1,\mu'}\geq0$. So we assume that $h^0(E)>0$, and let $S$ be as in Theorem \ref{thm:stratified_bound} above. As $E$ is not generically globally generated, $\rk(S)\leq r-1$. By Lemma \ref{lem:monotonicity_of_beta_1}, the assertion is clear.
\end{proof}

\begin{remark}
	Let $\mu:=\mumax(E)$. Then in the notation of the above theorem and the corollary, we must have $\lfloor\mu'\rfloor=\lfloor\mu\rfloor$ (otherwise we have $\mu'<\lfloor\mu\rfloor\leq\mu$, and since $\lfloor\mu\rfloor=j\lfloor\mu\rfloor/j$ for any $j\in\{2,4,\ldots,2(r-1),2r\}$, we have a contradiction to the choice of $\mu'$).
\end{remark}

\section{Investigating maximality - balanced case}\label{sec:maximal_balanced}

Following \cite{CHR25}, we consider the question of characterizing $\mu_H$-semistable sheaves that achieve the bound in Theorem \ref{thm:bound_sections}. The answer we find is analogous to the case of $\PP^2$ \textendash\,\,all such sheaves are \emph{balanced twisted Steiner-like}, given by nice resolutions. They are general elements of their moduli spaces (see \cite{Gaeta}). We begin by investigating global generation and nonspeciality properties of sheaves achieving the bound, which then come in handy to characterize them using a Beilinson spectral sequence.

\begin{definition}\label{def:deficiency_maximality}
	A $\mu_H$-semistable sheaf $E$ on $\Quadric$ of rank $r\geq1$ and slope $\mu\geq-1$ is said to have \emph{deficiency} $\delta$ if $h^0(E)=\beta_{r,\mu}-\delta$. We call $E$ \emph{maximal} if it has deficiency $0$.
\end{definition}

Unless specified otherwise, whenever we talk about deficiency (and maximality) of a sheaf, we shall assume the sheaf to be $\mu_H$-semistable.

Our main result in this section is the following.

\begin{theorem}\label{thm:full_structure_thm_maximal_sheaves_with_vanishing_h1}
	Let $E$ be a maximal sheaf on $\Quadric$ of rank $r\geq1$ and slope $\mu\geq0$. Then $h^1(E)=h^2(E)=0$ and $c_1(E)=r(\alpha-1,\alpha-1)+(m,n)$ for some non-negative integers $m,n$, with $\alpha=\alpha_{\mu}$ and $m+n<2r$. Moreover, $E$ is a \emph{balanced twisted Steiner-like sheaf} (see Definition \ref{def:Steiner-like_sheaf}) having a resolution $$0\to \OO(\alpha-2,\alpha-1)^m\oplus\OO(\alpha-1,\alpha-2)^n\to \OO(\alpha-1,\alpha-1)^{r+m+n}\to  E\to 0,$$ and it is a globally generated vector bundle. Furthermore, $E$ also has a resolution of the form $$0\to \OO(-1,-1)^{\beta_{r,\mu-1}}\to \OO(-1,0)^{m+r(\alpha-1)+\beta_{r,\mu-1}}\oplus\OO(0,-1)^{n+r(\alpha-1)+\beta_{r,\mu-1}}\to \OO^{\beta_{r,\mu}}\to E\to 0.$$
\end{theorem}

The rest of this section develops material needed to prove the above theorem.

\subsection{Global generation}

In this subsection we discuss global generation properties of sheaves with small deficiency.

\begin{lemma}\label{lem:small_deficiency_implies_gen_gg}
	Let $E$ be a $\mu_H$-semistable sheaf on $\Quadric$ of rank $r\geq1$, slope $\mu\geq0$ and deficiency strictly less than $\alpha_{\mu}^2$. Then $E$ is generically globally generated.
\end{lemma}

\begin{proof}
	First of all, note that $E$ must have a non-zero global section. Suppose not. Then $$0=h^0(E)>\beta_{r,\mu}-\alpha_{\mu}^2,$$ whence $\beta_{r,\mu}<\alpha_{\mu}^2$. This is a contradiction, since $\beta_{r,\mu}\geq r\alpha_{\mu}^2\geq\alpha_{\mu}^2$.
	
	Now suppose $E$ is not generically globally generated. Let $S\subset E$ be the image of the canonical evaluation map $\OO_{\Quadric}\otimes H^0(E)\to E$, and let $r'$ be its rank. Let $\mu':=\mumax(S)$ and $\alpha':=\alpha_{\mu'}$. Then $\mu'\leq\mu$ and thus $\alpha'\leq\alpha$, where $\alpha:=\alpha_{\mu}$ for simplicity of notation.
	
	Suppose first that $\alpha'=\alpha$. Then \begin{align*}
		\beta_{r,\mu}-\beta_{r',\mu'}&\geq\beta_{r,\mu}-(r-1)\alpha(2\mu-\alpha+2)\\
		&=\alpha(2\mu-\alpha+2)\\
		&\geq\alpha(2(\alpha-1)-\alpha+2)\qquad\qquad\quad(\mbox{since }\mu\geq\alpha-1).\\
		&=\alpha^2.
	\end{align*} Since $S$ is globally generated, $\mu'\geq0>-1$ by Lemma \ref{lem:ggg_implies_a,b_nonnegative}. Thus, by Theorem \ref{thm:bound_sections}, $$h^0(E)=h^0(S)\leq\beta_{r',\mu'}\leq\beta_{r,\mu}-\alpha^2<h^0(E).$$ This is a contradiction.
	
	Next, suppose $\alpha'<\alpha$. We see that \begin{align*}
		\beta_{r',\mu'}&<r'\alpha'(\alpha'+2)\qquad\qquad\qquad\,(\mbox{since }\mu'<\alpha')\\
		&\leq r'(\alpha-1)(\alpha+1)\qquad\,\qquad(\mbox{since }\alpha'\leq\alpha-1)\\
		&=r'(\alpha^2-1)\\
		&<r'\alpha^2\\
		&\leq(r-1)\alpha^2\\
		&=r\alpha^2-\alpha^2\\
		&\leq\beta_{r,\mu}-\alpha^2\qquad\qquad\qquad\quad(\mbox{since }\alpha\leq\mu+1).
	\end{align*} We again have a contradiction as earlier.
\end{proof}

\begin{remark}\label{rmk:deficiency_less_than_alpha_square_is_sharp_for_gen_gg}
	The above result is sharp. For example, the sheaf $I_{p_1,\ldots,p_4}(1,1)$ has deficiency $\alpha_{\mu}^2$ and is not generically globally generated, where $p_1,\ldots,p_4$ are general points of $\Quadric$.
\end{remark}

In fact, we have a stronger global generation property for maximal sheaves.

\begin{proposition}\label{prop:small_deficiency_implies_gg_cdim_1}
	Let $E$ be a maximal sheaf on $\Quadric$ of rank $r\geq1$ and slope $\mu\geq0$. Then $E$ is globally generated in codimension $1$.
\end{proposition}

\begin{proof}
	Since $h^0(E)=\beta_{r,\mu}\geq r\alpha_{\mu}^2\geq1$, $E$ has at least one non-zero global section.
	
	Let $\alpha:=\alpha_{\mu}$. We shall induct on $\alpha$.
	
	First, suppose that $\alpha=1$, and that $E$ is not globally generated in codimension $1$. Let $S\subset E$ be the image of the canonical evaluation map $\OO_{\Quadric}\otimes H^0(E)\to E$, and let $\mu':=\mu_H(S)$. By Lemma \ref{lem:small_deficiency_implies_gen_gg}, we know that $E$ is generically globally generated. So $\rk(S)=r$. The cokernel $T$ of the inclusion $S\hookrightarrow E$ is a torsion sheaf with $1$-dimensional support. Thus, $c_1(E)\cdot H>c_1(S)\cdot H$. So we see that \begin{align*}
		\beta_{r,\mu}-\beta_{r,\mu'}&=r(2\mu+1)-r(2\mu'+1)\\
		&=2r(\mu-\mu')\\
		&=c_1(E)\cdot H-c_1(S)\cdot H\\
		&\geq1.
	\end{align*} Note that $\mumax(S)\leq\mu(E)<1$. Also, since $\mumin(S)$ is the slope of a quotient of $S$, and $S$ is globally generated, $\mumin(S)$ must be non-negative. Clearly, $\alpha_{\mu'}=1$. By Corollary \ref{cor:strengthened_bound}, $h^0(S)\leq\beta_{r,\mu'}$. Thus, $$h^0(E)=h^0(S)\leq\beta_{r,\mu'}\leq\beta_{r,\mu}-1<\beta_{r,\mu}=h^0(E).$$ This is a contradiction. So $E$ must be globally generated in codimension $1$.
	
	Now suppose $\alpha>1$. Consider the restriction exact sequence $$0\to E(-1,-1)\to E\to E|_C\to 0$$ for a general curve $C$ of class $(1,1)$. It gives us the long exact sequence $$0\to H^0(E(-1,-1))\to H^0(E)\to H^0(E|_C)\to H^1(E(-1,-1))\to H^1(E)\to 0,$$ where the last $0$ term is due to the vanishing of $H^1(E|_C)$ since $E|_C$ is globally generated (due to Lemma \ref{lem:small_deficiency_implies_gen_gg}). Let $\kappa':=h^1(E(-1,-1))-h^1(E)\geq0$. Then we have \begin{align*}
		h^0(E(-1,-1))&=h^0(E)-h^0(E|_C)+\kappa'\\
		&=\beta_{r,\mu}-(r+a+b)+\kappa'\\
		&=\beta_{r,\mu}-r(2\mu+1)+\kappa'\\
		&=r(\alpha-1)(2\mu-\alpha+1)+\kappa'\\
		&=\beta_{r,\mu-1}+\kappa'.
	\end{align*} By Theorem \ref{thm:bound_sections}, we know that $h^0(E(-1,-1))\leq\beta_{r,\mu-1}$, whence $\kappa'=0$. Thus, $E(-1,-1)$ is a maximal sheaf. By induction on $\alpha$, $E(-1,-1)$ is globally generated in codimension $1$. It follows that $E$ is globally generated in codimension $1$ as well.
\end{proof}

\subsection{Nonspeciality}

We now consider the question of nonspeciality of maximal sheaves. More precisely, we show that maximal sheaves have vanishing $h^1$. Since $\mu_H$-semistable sheaves on $\Quadric$ with slope $\geq-1$ do not have any $h^2$, it follows that maximal sheaves are \emph{nonspecial}, i.e., have only one non-zero cohomology group. In terms of Brill-Noether loci, this tells us that $B^{\chi(E)}$ is the entire moduli space $M(\rk(E),c_1(E),\chi(E))$ for any maximal sheaf $E$.

We first prove our result for vector bundles. In Subsection \ref{subsec:local_freeness_of_maximals} we show that any maximal sheaf is actually a vector bundle.

\begin{proposition}\label{prop:vanishing_h1_vbundles}
	Let $E$ be a maximal sheaf on $\Quadric$ of rank $r\geq1$ and slope $\mu\geq0$ that is also a vector bundle. Then $H^1(E)=0$.
\end{proposition}

\begin{proof}
	The proof is along the same lines as \cite[Proposition 3.1]{LePotier80} and \cite[Proposition 4.7]{CHR25}.
	
	Let $c_1(E)=(a,b)$ in what follows. Note that $E$ is globally generated in codimension $1$ by Proposition \ref{prop:small_deficiency_implies_gg_cdim_1}. As in the proof of that proposition, we get a map $$\varphi_C:H^1(E(-1,-1))\to H^1(E)$$ of vector spaces coming from the long exact sequence $$0\to H^0(E(-1,-1))\to H^0(E)\to H^0(E|_C)\to H^1(E(-1,-1))\xrightarrow{\varphi_C} H^1(E)\to H^1(E|_C)$$ induced by the restriction sequence for each curve $C$ of class $(1,1)$, and $\varphi_C$ is surjective if $C$ is smooth (here we can say this for \emph{each} smooth $C$ instead of just a general $C$ because $E$ is a vector bundle that is globally generated in codimension $1$, so that $H^1(E|_C)=0$ if $C$ is a smooth rational curve). If $C\in|\OO(1,1)|$ is a union of two fibers $C_1\in|\OO(1,0)|$ and $C_2\in|\OO(0,1)|$, we have the following exact sequence, where $p$ denotes the point of intersection of $C_1$ and $C_2$, and the second map takes any tuple $(s_1,s_2)$ of sections to $s_1(p)-s_2(p)$. \begin{align*}
		0\to H^0(E|_C)&\to H^0(E|_{C_1})\oplus H^0(E|_{C_2})\to H^0(E|_p)\to0.
	\end{align*} Thus, $$h^0(E|_C)=h^0(E|_{C_1})+h^0(E|_{C_2})-r=(r+b)+(r+a)-r=r+a+b,$$ which agrees with $h^0(E|_{C'})$ for any smooth member $C'$ of the linear series $|\OO(1,1)|$. Since the Euler characteristic is constant in flat families, it follows that $H^1(E|_C)=0$, whence $\varphi_C$ is surjective. To summarize, we have shown that $\varphi_C$ is surjective for each $C\in|\OO(1,1)|$. These maps fit together to give us a surjective map $$\varphi:\OO_{\PP^3}(-1)\otimes H^1(E(-1,-1))\to \OO_{\PP^3}\otimes H^1(E)$$ of vector bundles on $\PP^3=|\OO_{\Quadric}(1,1)|$. Once again, from the proof of Proposition \ref{prop:small_deficiency_implies_gg_cdim_1} we know that $h^1(E(-1,-1))=h^1(E)$. Thus, $\varphi$ must be an isomorphism. This is a contradiction unless $h^1(E)=0$.
\end{proof}

\subsection{Steiner-like sheaves}

In this subsection we study maximal sheaves that have vanishing first cohomology (in particular, this includes maximal sheaves that are vector bundles, due to Proposition \ref{prop:vanishing_h1_vbundles}). We show that they must be \emph{balanced twisted Steiner-like} vector bundles.

We start with a definition motivated by the notion of Steiner sheaves on $\PP^2$ which are given as cokernels $$0\to \OO_{\PP^2}(-1)^a\to \OO_{\PP^2}^{a+r}\to E\to 0.$$ See \cite{Brambilla05,Brambilla08,Huizenga12} for extensive treatments of Steiner sheaves.

\begin{definition}\label{def:Steiner-like_sheaf}
	\mbox{ }\\[-15pt]
	\begin{enumerate}
		\item A torsion-free sheaf $E$ on $\Quadric$ is said to be a \emph{Steiner-like sheaf} if it fits in an exact sequence of the form $$0\to \OO(-1,0)^a\oplus\OO(0,-1)^b\to \OO^{r+a+b}\to E\to 0$$ for some non-negative integers $r,a,b$ with $r\geq1$. Such a sheaf $E$ has rank $r$, first Chern class $(a,b)$ and slope $\frac{a+b}{2r}$.
		\item A torsion-free sheaf $E$ on $\Quadric$ is said to be a \emph{twisted Steiner-like sheaf} if it fits in an exact sequence of the form $$0\to \OO(k-1,\ell)^m\oplus\OO(k,\ell-1)^n\to \OO(k,\ell)^{r+m+n}\to E\to 0$$ for some integers $k,\ell$ and some non-negative integers $r,m,n$ with $r\geq1$ and $m+n<2r$. Such a sheaf $E$ has rank $r$, first Chern class $r(k,\ell)+(m,n)$ and slope $\frac{k+\ell}{2}+\frac{m+n}{2r}$.
		\item A twisted Steiner-like sheaf $E$ is said to be \emph{balanced} if $k=\ell$ in the above definition, i.e., it fits in an exact sequence of the form $$0\to \OO(k-1,k)^m\oplus\OO(k,k-1)^n\to \OO(k,k)^{r+m+n}\to E\to 0$$ for some integer $k$ and some non-negative integers $r,m,n$ with $r\geq1$ and $m+n<2r$. Such a sheaf $E$ has rank $r$, first Chern class $r(k,k)+(m,n)$ and slope $k+\frac{m+n}{2r}$.
	\end{enumerate}
\end{definition}

We first perform some elementary computations on these sheaves.

\begin{lemma}\label{lem:basic_calculations_twisted_Steiners}
	Consider integers $m,n,r,k$ with $m,n\geq0$, $r\geq1$, $k\geq-1$ and $m+n<2r$, and a balanced twisted Steiner-like sheaf $E$ given as a cokernel $$0\to \OO(k-1,k)^m\oplus\OO(k,k-1)^n\to \OO(k,k)^{r+m+n}\to E\to 0.$$ Then $\alpha_{\mu}=k+1$ and $h^0(E)=\chi(E)=\beta_{r,\mu}$, where $\mu:=\mu_H(E)=k+\frac{m+n}{2r}$. Also, $h^1(E)=0$.
\end{lemma}

\begin{proof}
	The $\mu$ and $\alpha_{\mu}$ computations are straightforward. Looking at the long exact sequence in cohomology induced by the given resolution of $E$, and recalling that $\OO(k-1,k)$, $\OO(k,k-1)$ and  $\OO(k,k)$ have vanishing higher cohomology, and $$h^0(\OO(k-1,k))=h^0(\OO(k,k-1))=k(k+1)$$ and $h^0(\OO(k,k))=(k+1)^2$, the other assertions follow.
\end{proof}

The semistability of sheaves (in particular, the balanced twisted Steiner-like sheaves) on $\Quadric$ is controlled by the Dr\'{e}zet-Le Potier-Rudakov surface. See \cite{Rudakov-Semistab,CH-Semistable_Hirzebruch} for details. In particular, by calculating the function $\operatorname{DLP}_{H,\mathcal{V}}$  (see \cite[Section 6]{CH-Semistable_Hirzebruch}) for rank $3$ exceptional characters $\mathcal{V}$, one finds that balanced twisted Steiner-like sheaves that are given by general resolutions and have discriminant larger than $5/9$ are $H$-Gieseker stable vector bundles, provided that the fractional part of both $k+(m/r)$ and $k+(n/r)$ is larger than $1/2$.

Lemma \ref{lem:basic_calculations_twisted_Steiners} shows that if semistable, then the balanced twisted Steiner-like sheaves are maximal in the sense of Definition \ref{def:deficiency_maximality}. We now investigate the converse and obtain a positive answer. In other words, we show that any maximal sheaf must be a balanced twisted Steiner-like vector bundle. We begin with a preliminary result.

\begin{proposition}\label{prop:max-sections_long_resolution}
	Let $E$ be a maximal sheaf on $\Quadric$ of rank $r\geq1$ and slope $\mu\geq0$. Suppose $H^1(E)=0$. Write $c_1(E)=(a,b)$ for some integers $a,b$. Then $E$ has a resolution of the form $$0\to \OO(-1,-1)^{\beta_{r,\mu-1}}\to \OO(-1,0)^{a+\beta_{r,\mu-1}}\oplus\OO(0,-1)^{b+\beta_{r,\mu-1}}\to \OO^{\beta_{r,\mu}}\to E\to 0.$$ In particular, $E$ is globally generated.
\end{proposition}

\begin{proof}
	As in the proof of Proposition \ref{prop:small_deficiency_implies_gg_cdim_1} above, we can calculate $h^0(E(-1,-1))=\beta_{r,\mu-1}+\kappa'$, where $\kappa':=h^1(E(-1,-1))-h^1(E)=h^1(E(-1,-1))\geq0$. Due to Theorem \ref{thm:bound_sections}, $\kappa'=0$. Thus we have $h^0(E(-1,-1))=\beta_{r,\mu-1}$ and $h^1(E(-1,-1))=0$. Moreover, by Serre duality and $\mu_H$-semistability, $$h^2(E)=h^2(E(-1,0))=h^2(E(0,-1))=h^2(E(-1,-1))=0.$$ We have a Beilinson spectral sequence (\cite[Proposition 5.1]{Drezet(BeilinsonSpecSeqn)}) that converges to $E$ in degree $0$ and to $0$ elsewhere, and whose $E^{p,q}_1$-page takes the form $$\xymatrix{
		\OO(-1,-1)^{h^1(E(-1,-1))} \ar[r] & \OO(-1,0)^{h^1(E(-1,0))} \oplus \OO(0,-1)^{h^1(E(0,-1))}  \ar[r] & \OO^{h^1(E)}  \\
		\OO(-1,-1)^{h^0(E(-1,-1))} \ar[r] & \OO(-1,0)^{h^0(E(-1,0))} \oplus \OO(0,-1)^{h^0(E(0,-1))}  \ar[r] & \OO^{h^0(E)}. \\
	}$$ Plugging in the values of the known exponents, the $E^{p,q}_1$-page is $$\xymatrix{
		0 \ar[r] & \OO(-1,0)^{h^1(E(-1,0))} \oplus \OO(0,-1)^{h^1(E(0,-1))}  \ar[r] & 0  \\
		\OO(-1,-1)^{\beta_{r,\mu-1}} \ar[r] & \OO(-1,0)^{h^0(E(-1,0))} \oplus \OO(0,-1)^{h^0(E(0,-1))}  \ar[r] & \OO^{\beta_{r,\mu}}. \\
	}$$ Since the above spectral sequence converges to $E$, we obtain a canonical map $\OO^{\beta_{r,\mu}}\to E$ (which is just the canonical evaluation map) whose cokernel is the cokernel of the top left map shown above, which is just $$\OO(-1,0)^{h^1(E(-1,0))} \oplus \OO(0,-1)^{h^1(E(0,-1))}.$$ Since $E$ is generically globally generated (by Lemma \ref{lem:small_deficiency_implies_gen_gg}), this cokernel must be either zero or torsion. Thus, we must have $h^1(E(-1,0))=h^1(E(0,-1))=0$. From the convergence of the above spectral sequence to $E$, We get a resolution $$0\to \OO(-1,-1)^{\beta_{r,\mu-1}}\to \OO(-1,0)^{h^0(E(-1,0))}\oplus\OO(0,-1)^{h^0(E(0,-1))}\to \OO^{\beta_{r,\mu}}\to E\to 0$$ of $E$. Computing first Chern classes, we see that $$h^0(E(-1,0))=a+\beta_{r,\mu-1}\quad\mbox{ and }\quad h^0(E(0,-1))=b+\beta_{r,\mu-1}.$$
\end{proof}

As a corollary, we get the following result.

\begin{proposition}\label{prop:max-sections_are_twisted_Steiners}
	Let $E$ be a maximal sheaf on $\Quadric$ of rank $r\geq1$ and slope $\mu\geq0$. Suppose $H^1(E)=0$. Then $c_1(E)=r(\alpha-1,\alpha-1)+(m,n)$ for some non-negative integers $m$ and $n$, with $\alpha:=\alpha_{\mu}$ and $m+n<2r$. Moreover, $E$ is a balanced twisted Steiner-like sheaf having a resolution $$0\to \OO(\alpha-2,\alpha-1)^m\oplus\OO(\alpha-1,\alpha-2)^n\to\OO(\alpha-1,\alpha-1)^{r+m+n}\to E\to 0.$$
\end{proposition}

\begin{proof}
	We shall induct on $\alpha$. Write $c_1(E)=(a,b)$ for some integers $a$ and $b$.
	
	First suppose $\alpha=1$. Noting that $\beta_{r,\mu}=r+a+b$ and $\beta_{r,\mu-1}=0$, Proposition \ref{prop:max-sections_long_resolution} gives us a resolution $$0\to \OO(-1,0)^a\oplus\OO(0,-1)^b\to \OO^{r+a+b}\to E\to 0,$$ thus proving that $E$ is a Steiner-like sheaf. Moreover, $c_1(E)=(m,n)+r(\alpha-1,\alpha-1)$ for $m:=a\geq0$ and $n:=b\geq0$. Note that $$\frac{m+n}{2r}=\frac{a+b}{2r}=\mu<1.$$
	
	Now suppose $\alpha>1$. We have shown while proving Proposition \ref{prop:small_deficiency_implies_gg_cdim_1} that $E(-1,-1)$ is also a maximal sheaf. Moreover, in that same proof, we also showed that $\kappa'=0$, where $\kappa':=h^1(E(-1,-1))-h^1(E)$. Thus, $h^1(E(-1,-1))=0$. By induction on $\alpha$, it follows that $c_1(E(-1,-1))=(m,n)+r(\alpha-2,\alpha-2)$ for some non-negative integers $m$ and $n$ with $m+n<2r$, and there is an exact sequence $$0\to \OO(\alpha-3,\alpha-2)^m\oplus\OO(\alpha-2,\alpha-3)^n\to \OO(\alpha-2,\alpha-2)^{r+m+n}\to E(-1,-1)\to 0.$$ Twisting by $\OO(1,1)$ we get the exact sequence $$0\to \OO(\alpha-2,\alpha-1)^m\oplus\OO(\alpha-1,\alpha-2)^n\to\OO(\alpha-1,\alpha-1)^{r+m+n}\to E\to 0.$$ Clearly, $c_1(E)=(m,n)+r(\alpha-1,\alpha-1)$.
\end{proof}

\subsection{Local freeness}\label{subsec:local_freeness_of_maximals}

We now show that the assumption of $E$ being a vector bundle in Proposition \ref{prop:vanishing_h1_vbundles} is not really necessary, as any maximal sheaf is in fact a vector bundle. We start by proving a more general result.

\begin{proposition}\label{prop:maximal_section_with_vanishing_h1_is_vbundle}
	Let $E$ be a $\mu_H$-semistable sheaf on $\Quadric$ of rank $r\geq1$ and slope $\mu\geq-1$. Suppose $H^1(E)=0$. Then the number of points where $E$ is singular (i.e., fails to be locally free) is at most the deficiency of $E$. In particular, if $E$ is a maximal sheaf with vanishing $H^1$, then $E$ is a vector bundle.
\end{proposition}

\begin{proof}
	If $E$ is a vector bundle, we have nothing to prove. So, suppose $E$ is not a vector bundle and thus has some singularities. Let $E^{**}$ be the double dual of $E$, and let $T$ be the cokernel of the inclusion $E\hookrightarrow E^{**}$. Then $T$ is non-zero, and has zero-dimensional support. We shall induct on the length of $T$. There is an exact sequence $$0\to \OO_p\to T\xrightarrow{\varphi} T'\to 0,$$ where $p$ is a singularity of $E$. Let $E'$ be the kernel of the composition $E^{**}\to T\xrightarrow{\varphi} T'$.  It is again $\mu_H$-semistable, and it has the same double dual as $E$, i.e., $(E')^{**} = E^{**}$. There is also a natural inclusion $i:E\hookrightarrow E'$.  Using the snake lemma on the commutative diagram $$\xymatrix{0\ar[r]&E\ar[r]\ar[d]_{i} & E^{**} \ar[r]\ar@{=}[d]& T \ar[d]_{\varphi} \ar[r]& 0\\
		0\ar[r]&E'\ar[r] & E^{**} \ar[r] & T'\ar[r]& 0
	}
	$$ with exact rows, we get an exact sequence $$0\to \ker\varphi\to \coker i\to 0.$$ This shows that $\coker i\cong\OO_p$ and thus there is an exact sequence $$0\to E\xrightarrow{i} E'\to \OO_p\to 0.$$  It induces the exact sequence $$0\to H^0(E)\to H^0(E')\to H^0(\OO_p)\to 0$$ since $H^1(E)=0$ by assumption. The rank-nullity theorem shows that $$h^0(E')=h^0(E)+1=(\beta_{r,\mu}-\delta)+1=\beta_{r,\mu}-(\delta-1),$$ where $\delta$ is the deficiency of $E$. Since $E'$ has the same rank and slope as $E$, this means that the deficiency of $E'$ is one less than that of $E$. Note that $H^1(E')$ clearly vanishes since $H^1(E)$ and $H^1(\OO_p)$ do. Since $\length(T')=\length(T)-1$, the exact sequence $$0\to E'\to (E')^{**}\to T'\to 0$$ shows that we can proceed by induction. Thus, if $E'$ is not a vector bundle, we can repeat this process to resolve one singularity (counting multiplicities) of $E'$ and obtain a $\mu_H$-semistable sheaf satisfying the hypotheses of this proposition but having deficiency one less than $E'$, and so on. This process cannot go on further once we have obtained a sheaf of deficiency $0$. It follows that $\length(T)\leq\delta$.
\end{proof}

We are now in a position to show that maximal sheaves are always vector bundles. In other words, we now drop the cohomology vanishing hypothesis made above.

\begin{proposition}\label{prop:maximal_section_with_gen_gg_is_vbundle}
	Any maximal sheaf on $\Quadric$ of rank $r\geq1$ and slope $\mu\geq0$ is a vector bundle.
\end{proposition}

\begin{proof}
	Suppose the assertion is false. Then choose a maximal sheaf $E$ satisfying the conditions of this proposition such that $E$ is not a vector bundle, and has the minimal number of singularities (counting multiplicities) among all sheaves that satisfy the hypotheses of this proposition but are not vector bundles. Let $E^{**}$ be the double dual of $E$, and let $T$ be the cokernel of the inclusion $E\hookrightarrow E^{**}$. Then $T$ is non-zero, and has zero-dimensional support. Let $p, T'$ and $E'$ be as in the proof of Proposition \ref{prop:maximal_section_with_vanishing_h1_is_vbundle} above, and we have an exact sequence \begin{equation}\label{eq:short_exact_sequence_01}
		0\to E\to E'\to \OO_p\to 0.
	\end{equation} Thus, $h^0(E') \geq h^0(E)$. Since $E'$ has the same rank and slope as $E$, it follows that $E'$ is also a maximal sheaf. Now the exact sequence $$0\to E'\to (E')^{**}\to T'\to 0$$ shows us that $E'$ has a smaller number of singularities (counting multiplicities) than $E$. The way that $E$ was chosen, it follows that $E'$ must be a vector bundle. Proposition \ref{prop:vanishing_h1_vbundles} tells us that $H^1(E')=0$. Also, Proposition \ref{prop:max-sections_long_resolution} implies that $E'$ is globally generated, whence the map $H^0(E')\rightarrow H^0(\OO_p)$ is surjective. The long exact sequence $$0\to H^0(E)\to H^0(E')\to H^0(\OO_p)\to H^1(E)\to 0$$ induced by \eqref{eq:short_exact_sequence_01} then shows that $H^1(E)=0$. Proposition \ref{prop:maximal_section_with_vanishing_h1_is_vbundle} now leads to a contradiction.
\end{proof}

We can now complete the proof of Theorem \ref{thm:full_structure_thm_maximal_sheaves_with_vanishing_h1}.

\begin{proof}[Proof of Theorem \ref{thm:full_structure_thm_maximal_sheaves_with_vanishing_h1}]
	Clear from Propositions \ref{prop:maximal_section_with_gen_gg_is_vbundle},  \ref{prop:vanishing_h1_vbundles}, \ref{prop:max-sections_are_twisted_Steiners} and \ref{prop:max-sections_long_resolution}.
\end{proof}

We also have the following corollary regarding Brill-Noether loci, showing that $B^{\beta_{r,\mu}}$ is either the entire moduli space, or empty. Furthermore, if $\chi=\beta_{r,\mu}$, the Brill-Noether loci are either all of the moduli space, or empty.

\begin{corollary}\label{cor:BN_loci}
	Suppose $M:=M_{\Quadric,H}(r,c_1,\chi)$ is a nonempty moduli space of $H$-Gieseker semistable sheaves on $\Quadric$ of rank $r\geq1$, first Chern class $c_1$ and Euler characteristic $\chi$, with $$\mu:=\frac{c_1\cdot H}{rH^2}\geq0.$$ Let $\beta:=\beta_{r,\mu}$. Then the Brill-Noether locus $B^{\beta}\subseteq M$ is nonempty if and only if $\chi=\beta$. In the case that $\chi=\beta$, the locus $B^k$ is all of $M$ if and only if $k\leq\beta$, and is empty otherwise.
\end{corollary}

\begin{proof}
	Suppose $B^{\beta}$ is nonempty. If $E$ is an element of this locus, then $h^0(E)=\beta$. Also, $h^1(E)=0$ by Propositions \ref{prop:vanishing_h1_vbundles} and \ref{prop:maximal_section_with_gen_gg_is_vbundle}. Of course, $h^2(E)=0$ since $\mu(E)\geq0$. Thus, $\chi=\chi(E)=\beta$.
	
	Conversely, suppose $\chi=\beta$, and let $E$ be any element of $M$. Since $h^2(E)=0$, it follows that $h^0(E)\geq\chi(E)=\beta$. So $B^{\beta}=M$ (whence $B^k=M$ for any $k\leq\beta$). Since $h^0(E)\leq\beta$, it is obvious that $B^k$ is empty for $k\geq\beta+1$.
\end{proof}

\section{Investigating maximality - unbalanced case}\label{sec:maximal_unbalanced}

We now study sheaves that attain the bound given by Theorem \ref{thm:better_bound_sections_stratified}. We find that such sheaves have good global generation properties, and are often just twisted Steiner-like bundles. Before that, we quickly demonstrate the converse, i.e., that twisted Steiner-like bundles achieve the bound.

\begin{example}
	Consider a balanced twisted Steiner-like bundle $S$ on $\Quadric$, with $m,n,r,k$ as in Definition \ref{def:Steiner-like_sheaf}(iii). Assume $k\geq0$, and $0\leq n-m<r$. Let $\ell$ be a non-negative integer, and define $E:=S(0,\ell)$. Then $\mu_H(E)=\mu_H(S)+\ell/2$, i.e., $$\mu':=\mu_H(S)=\mu_H(E)-\frac{(\ell+1)}{2}+\frac{1}{2}.$$ Note that the ``$j$'' appearing in Theorem \ref{thm:better_bound_sections_stratified} for the sheaf $E$ is equal to $\ell$. There is an exact sequence $$0\to \OO(k-1,k+\ell)^m\oplus\OO(k,k-1+\ell)^n\to \OO(k,k+\ell)^{r+m+n}\to E\to 0,$$ which shows that $$h^0(E)=(r+m+n)(k+1)(k+\ell+1)-mk(k+\ell+1)-n(k+1)(k+\ell)=\beta_{r,\mu'}+\ell(r+(rk+m)).$$ Hence, $h^0(E)$ equals the bound in Theorem \ref{thm:better_bound_sections_stratified}.
\end{example}

\begin{theorem}\label{thm:strongly_maximal_structure}
	Let $E$ be a $\mu_H$-semistable sheaf of rank $r\geq1$ on $\Quadric$ with $c_1(E)=(a,b)$. Without loss of generality, assume $b\geq a$. Suppose $E$ is generically globally generated, and $j$ is the largest integer with $$0\leq j\leq \Bigg\lfloor\frac{b-a}{r}\Bigg\rfloor$$ for which $E\!\left(0,-j\right)$ is generically globally generated. Let $$\mu'':=\mu_H(E)-\frac{(j+1)}{2},$$ and suppose that $$h^0(E)=\beta_{r,\mu''+1/2}+j(r+a).$$ Then $h^1(E)=h^2(E)=0$, and $E$ is a globally generated vector bundle with $E\cong S(0,j)$ for some balanced twisted Steiner-like bundle $S$.
\end{theorem}

\begin{proof}
	From the proof of Theorem \ref{thm:better_bound_sections_stratified}, we see that we must have equalities \begin{align*}
		h^0(E)&=h^0(E(0,-1))+r+a\\
		&=h^0(E(0,-2))+2(r+a)\\
		&=\cdots\\
		&=h^0\!\left(E\!\left(0,-j\right)\right)+j(r+a)\\
		&=\beta_{r,\mu''+1/2}+j(r+a).\end{align*}
	So $h^0\left(E\!\left(0,-j\right)\right)=\beta_{r,\mu''+1/2}$. Since $E\!\left(0,-j\right)$ is generically globally generated, $\mu(E\!\left(0,-j\right))\geq0$ by Lemma \ref{lem:ggg_implies_a,b_nonnegative}. Theorem \ref{thm:full_structure_thm_maximal_sheaves_with_vanishing_h1} then shows that $E\!\left(0,-j\right)$ is a balanced twisted Steiner-like bundle, and the vanishing of higher cohomology of $E$ follows from twisting the first resolution in that theorem.
\end{proof}

On the other hand, if $E$ achieves the bound in Theorem \ref{thm:better_bound_sections_stratified} due to $h^0(E)$ being equal to $$\beta_{r,\mu''}+(j+1)(r+a),$$ we have the following result which shows that $E$ can be twisted down a lot and still retain global generation properties, and $E$ is often a twisted Steiner-like bundle.

\begin{theorem}\label{thm:strongly_maximal_structure_2}
	Let $E$ be a $\mu_H$-semistable sheaf of rank $r\geq1$ on $\Quadric$ with $c_1(E)=(a,b)$. Without loss of generality, assume $b\geq a$. Suppose $E$ is generically globally generated, and $j$ is the largest integer with $$0\leq j\leq \Bigg\lfloor\frac{b-a}{r}\Bigg\rfloor$$ for which $E\!\left(0,-j\right)$ is generically globally generated. Let $$\mu'':=\mu_H(E)-\frac{(j+1)}{2},$$ and suppose that $$h^0(E)=\beta_{r,\mu''}+(j+1)(r+a).$$ Then $j$ must be equal to $\lfloor(b-a)/r\rfloor$, i.e., $E\!\left(0,-\lfloor(b-a)/r\rfloor\right)$ must be generically globally generated.
	
	Furthermore, suppose $a\geq r/2$. Then $h^1(E)=h^2(E)=0$, and $E$ is a globally generated vector bundle with $E\cong S(0,\lfloor(b-a)/r\rfloor+1)$ for some balanced twisted Steiner-like bundle $S$.
\end{theorem}

\begin{proof}
	Similarly as in the proof of Theorem \ref{thm:strongly_maximal_structure} above, we get $h^0\left(E\!\left(0,-j-1\right)\right)=\beta_{r,\mu''}$.
	
	Suppose $j$ is strictly smaller than $\lfloor(b-a)/r\rfloor$. Now suppose $\mu(E\!\left(0,-j-1\right))<0$, i.e., $$\frac{a+b}{2r}-\frac{(j+1)}{2}<0.$$ Then $$\frac{a+b}{r}-1<j\leq\Bigg\lfloor\frac{b-a}{r}\Bigg\rfloor-1\leq\frac{b-a}{r}-1.$$ So $$\frac{a+b}{r}<\frac{b-a}{r},$$ which implies that $a<0$. This is impossible, since $E$ is generically globally generated and we have Lemma \ref{lem:ggg_implies_a,b_nonnegative}. So $\mu(E\!\left(0,-j-1\right))\geq0$. Lemma \ref{lem:small_deficiency_implies_gen_gg} tells us that $E(0,-j-1)$ is generically globally generated, since its deficiency is zero. This contradicts the definition of $j$. Hence, $j=\lfloor(b-a)/r\rfloor$.
	 
	 Now if $a\geq r/2$, then $\mu''\geq0$. Theorem \ref{thm:full_structure_thm_maximal_sheaves_with_vanishing_h1} proves the remaining assertions.
\end{proof}

\section{Semistability of extensions}\label{sec:semistab_extn}

We take a detour in this section to prove the semistability of extensions of stable vector bundles on elliptic curves and del Pezzo surfaces. This will be used in the following section to show the sharpness of the bound in Theorem \ref{thm:stratified_bound}.

\subsection{Extensions on curves}

We first investigate the semistability of extensions of stable bundles on elliptic curves, and find positive answers.

\begin{theorem}\label{thm-elliptic}
	Let $C$ be an elliptic curve with an arbitrary polarization. Let $A$ and $B$ be stable bundles on $C$ with $\mu(A)<\mu(B)$. Let $V$ be given by a general extension of the form $$0\to A\to V\to B\to 0.$$  Then\ $V$ is semistable.
\end{theorem}

\begin{proof}
	The idea of the proof is to show that the family of all non-split extensions  is a complete family of sheaves with fixed rank and determinant.  Then since the stack of coherent sheaves on $C$ with fixed rank and determinant is irreducible and the semistable sheaves are a dense open substack, it follows that the general extension is semistable.
	
	Let $S = \Ext^1_C(B,A)$ and let $V_s/S$ be the universal extension sheaf. For $0\neq s\in S$, the Kodaira-Spencer map $$T_s S = \Ext^1_C(B,A)\to \Ext^1_C(V_s,V_s)$$ factors as the composition of natural maps $$\Ext^1_C(B,A) \fto\alpha \Ext^1_C(B,V_s) \fto\beta \Ext^1_C(V_s,V_s).$$  These maps fit into the exact sequences $$\Ext^1_C(B,A)\fto\alpha \Ext^1_C(B,V_s)\to \Ext^1_C(B,B)\to 0$$ $$\Ext^1_C(B,V_s)\fto\beta \Ext^1_C(V_s,V_s)\to \Ext^1_C(A,V_s).$$  
	
	By Serre duality, $\Ext^1_C(A,V_s) \cong \Hom(V_s,A)^*$. Note that $\Hom(V_s,A)$ fits into an exact sequence $$0\to \Hom(B,A)\to \Hom(V_s,A)\to \Hom(A,A)\fto\delta \Ext^1_C(B,A).$$ Since $\mu(A)<\mu(B)$,  $\Hom(B,A) = 0$ by semistability. Since stable bundles are simple, we have $\Hom(A,A) = \CC\cdot\id$. Moreover, $\delta$ is injective because $\delta(\id) = s\neq0$.  So $\Hom(V_s,A) = 0$, whence $\beta$ is surjective.
	
	On the other hand, $\ext^1_C(B,B)=\hom(B,B)=1$ by Serre duality and stability, so the image of $\alpha$ has codimension $1$ in $\Ext^1_C(B,V_s)$.  Hence, the image of $T_s S$ in $\Ext^1_C(V_s,V_s)$ under the Kodaira-Spencer map has codimension at most $1$ in $\Ext^1_C(V_s,V_s)$.  Since all the sheaves parameterized by $S$ have fixed determinant, this image lies in the codimension $1$ subspace $\Ext^1_0(V_s,V_s)$ of traceless extensions.  Therefore the image equals $\Ext^1_0(V_s,V_s)$, which is the tangent space to the stack of coherent sheaves on $C$ with fixed determinant at $[V_s]$. So the family is complete at $s$.
\end{proof}

\subsection{Restrictions of extensions}

We now observe that general extensions on a surface restrict to general extensions on curves on the surface.

\begin{lemma}\label{lem-extend}
	Let $A$ be a sheaf on a smooth polarized surface $(X,L)$ which is locally free along a curve $C\subset X$. Let $M_{X,L}(r,c_1,\chi)$ be a moduli space of semistable sheaves on $X$ with $\chi \ll 0$ (depending on $A,r,c_1$). Suppose $M$ is an irreducible component of this moduli space such that the general element of $M$ is stable on $X$ and locally free along $C$. If $B\in M$ is general, then a general extension $$0\to A\to V\to B\to 0$$ induces a general extension $$0\to A|_C\to V|_C\to B|_C\to 0.$$
\end{lemma}

\begin{proof}
	If $r\geq2$, then due to the assumption $\chi\ll 0$, O'Grady's theorems (see \cite{O'Grady96}) imply that $M$ is itself irreducible, and $B$ is a $\mu_L$-stable vector bundle. It follows that we have natural isomorphisms $$\Ext^1_X(B,A|_C) \cong H^1((B^*\otimes A)|_C) \cong \Ext^1_C(B|_C,A|_C).$$ Given an extension class $s\in \Ext^1_X(B,A)$ defining $V$, the extension class $s'\in \Ext^1_C(B|_C,A|_C)$ defining $V|_C$ is given by the image of $s$ under the natural map$$\Ext^1_X(B,A)\to\Ext^1_X(B,A|_C)\cong\Ext^1_C(B|_C,A|_C).$$ This map fits into the exact sequence $$\Ext^1_X(B,A)\to \Ext^1_X(B,A|_C)\to \Ext^2_X(B,A(-C)).$$  Thus if $\Ext^2_X(B,A(-C))=0$, then a general extension on $X$ induces a general extension on $C$.
	
	We have $\ext^2_X(B,A(-C)) = \hom(A,B(C+K_X))$.  By semicontinuity, it suffices to produce a single $B$ which has $\hom(A,B(C+K_X))=0$. Suppose $\hom(A,B(C+K_X))\neq0$. If we replace $B$ by an elementary modification $$0\to B'\to B\to \OO_p\to 0$$ for some $p\in X$, then $\chi(B') = \chi(B)-1,$ $B'$ is again $\mu$-stable, and $\Hom(A,B'(C+K_X))$ is a vector subspace of $\Hom(A,B(C+K_X))$.   The map $B\to \OO_p$ depends on a choice of a kernel hyperplane $H_p\subset B_p$ in the fiber $B_p$.  If $p$ and the hyperplane $H_p$ are both chosen generally, then the map $\Hom(A,B(C+K_X))\to \Hom(A,\OO_p(C+K_X))$ will be nonzero, so that $$\Hom(A,B'(C+K_X))\subsetneq \Hom(A,B(C+K_X))$$ is a \emph{proper} subspace.  Repeatedly applying elementary modifications eventually yields a $\mu$-stable sheaf $B''$ with smaller Euler characteristic such that $\hom(A,B''(C+K_X))=0$.
	
	Finally, if $r=1$, then we consider the exact sequence $$0\to B\to B^{**}\to T\to 0,$$ where $B^{**}$ denotes the double dual of $B$. By considering the local-to-global $\Ext$ spectral sequence, we see that $\Ext^i_X(E,F)=0$ for all $i\geq0$, for any $\OO_X$-modules $E$ and $F$ on $X$ with disjoint supports. Thus, in particular, $\Ext^1_X(T,A|_C)=\Ext^2_X(T,A|_C)=0$. So we have a natural isomorphism $\Ext^1_X(B^{**},A|_C)\xrightarrow{\cong}\Ext^1_X(B,A|_C)$. Since $B^{**}$ is a vector bundle, it follows as before that $$\Ext^1_X(B^{**},A|_C)\cong\Ext^1_C(B^{**}|_C,A|_C)\cong\Ext^1_C(B|_C,A|_C).$$ The rest of the proof follows exactly as above.
\end{proof}

\subsection{Semistability of restrictions}

Given a surface $X$ with irregularity $0$ and an anticanonical elliptic curve $C$ on the surface, we now prove the semistability of the restriction of the general stable bundle on $X$ to $C$.

\begin{lemma}\label{lem-restrict}
	Let $(X,L)$ be a smooth polarized surface with irregularity $h^1(\OO_X)=0$ such that $|-K_X|$ contains a smooth elliptic curve $C$. Let $M$ be any irreducible component of a moduli space of sheaves on $X$ where the general sheaf is stable on $X$ and locally free along $C$. Then the general sheaf $V\in M$ restricts to a semistable bundle $V|_C$ on $C$.
\end{lemma}

\begin{proof}
	The proof is similar to that of Theorem \ref{thm-elliptic}. We show that if $V_s/S$ is a complete family of stable sheaves on $X$ that are locally free along $C$, then $V_s|_C/S$ is a complete family of sheaves with fixed determinant on $C$.  The Kodaira-Spencer map for the restricted family $$T_s S \to \Ext_C^1(V_s|_C,V_s|_C)$$ factors through the Kodaira-Spencer  map of the original family as $$T_s S \to \Ext_X^1(V_s,V_s)\to \Ext_C^1(V_s|_C,V_s|_C).$$  Here the first map is surjective.  We have an isomorphism $\Ext^1_C(V_s|_C,V_s|_C) \cong \Ext^1_X (V_s,V_s|_C)$, and thus the second map fits in the exact sequence $$\Ext^1_X(V_s,V_s)\to \Ext_X^1(V_s,V_s|_C)\to \Ext_X^2(V_s,V_s(-C))\to 0.$$ The last $0$ above is because $\Ext^2_X(V_s,V_s)\cong\Hom(V_s,V_s(-C))^*=0$ due to semistability. We also have $\ext^2_X(V_s,V_s(-C)) = \hom(V_s,V_s) = 1$, since $V_s$ is stable. So the image of $T_sS\to \Ext^1_C(V_s|_C,V_s|_C)$ has codimension $1$.  Since the sheaves $V_s|_C$ all have the same determinant, the image equals the subspace of traceless extensions, and $V_s|_C$\ is a complete family of sheaves with fixed determinant.
\end{proof}

\subsection{Extensions on surfaces}

Consider a del Pezzo surface $X$. Then $-K_X$ is ample and $|-K_X|$ contains a smooth elliptic curve $C$. Also, $h^1(\OO_X)=0$. Using the results above, we now show that the general extension of stable bundles on $X$ are semistable under some coprimality assumptions.

By Walter's theorem (see \cite{Walter98}), the moduli space of semistable sheaves on $X$ with fixed invariants is irreducible, if nonempty. Thus, it makes sense to talk about the general element of such a moduli space. The general such element is locally free along $C$ (it is a vector bundle if the rank is at least $2$, and it is a twist of an ideal sheaf $I_Z$ of a general $0$-dimensional subscheme $Z\subset X$ if the rank is $1$).

\begin{theorem}\label{thm-ss_ext_elliptic}
	Let $X$ be a del Pezzo surface with anticanonical polarization, and let $C\in|-K_X|$ be a smooth elliptic curve with an arbitrary polarization. Suppose $M'=M_{X,-K_X}(r',c_1',\chi')$ and $M''=M_{X,-K_X}(r'',c_1'',\chi'')$ are moduli spaces of semistable sheaves on $X$, with $\chi''\ll0$. Let $A$ and $B$ be general elements of $M'$ and $M''$ respectively. Assume that $A$ is stable on $X$, that $$\mu(A|_C)<\mu(B|_C),$$ and that $\gcd(r',c_1'\!\cdot\!C)=\gcd(r'',c_1''\!\cdot\!C)=1$. Then for a general extension $$0\to A\to V\to B\to 0,$$ $V$ is $\mu_{-K_X}$-semistable. If furthermore $\gcd(r'+r'',(c_1'+c_1'')\!\cdot\!C)=1$, then $V$ is $\mu_{-K_X}$-stable.
\end{theorem}

\begin{proof}
	By Lemma \ref{lem-restrict} and the coprimality assumptions, both $A|_C$ and $B|_C$ are stable.  By Lemma \ref{lem-extend}, $V|_C$ is given by a general extension of $B|_C$ by $A|_C$.  Theorem \ref{thm-elliptic} then shows that $V|_C$ is semistable, whence $V$ is $\mu_{-K_X}$-semistable.  If furthermore $r'+r''$ and $(c_1'+c_1'')\!\cdot\!C$\ are coprime, then $V|_C$ is stable and $V$ is $\mu_{-K_X}$-stable.
\end{proof}

\section{Investigating maximality - non-globally generated case}\label{sec:maximal_non_gg}

We now take up the question of $\mu_H$-semistable sheaves achieving the bounds from Theorem \ref{thm:stratified_bound} and Corollary \ref{cor:bound_sections_non_gen_gg}. We first prove a necessary condition for such sheaves with small slope, and then use Theorem \ref{thm-ss_ext_elliptic} to illustrate the sharpness of our bound.

\begin{theorem}\label{thm:non_gen_gg_maximal_small_slope_more_general}
	Let $E$ be a $\mu_H$-semistable sheaf of rank $r\geq1$ on $\Quadric$ with $\mu:=\mu_H(E)\in[0,1)$ and $h^0(E)>0$. Let $S\subset E$ be the image of the canonical evaluation map $\OO_{\Quadric}\otimes H^0(E)\to E$, and let $\mu'\in\left[0,\mu\right]$ be the largest rational number that can be written with denominator belonging to the set $\{2,4,\ldots,2\rk(S)\}$. Further suppose that $$h^0(E)=\beta_{\rk(S),\mu'}.$$ Then $S$ is a balanced twisted Steiner-like bundle of slope $\mu'$, and $E$ fits in an exact sequence of the form $$0\to S\to E\to Q\to 0$$ with $h^0(Q)=0$.
\end{theorem}

\begin{proof}
	As in the proof of Proposition \ref{prop:small_deficiency_implies_gg_cdim_1}, we see that $h^0(S)\leq\beta_{\rk(S),\mu_H(S)}$ by virtue of Corollary \ref{cor:strengthened_bound}. Since $\mu_H(S)\leq\mu'$ (since $E$ is $\mu_H$-semistable), it follows that $\beta_{\rk(S),\mu_H(S)}\leq\beta_{\rk(S),\mu'}$, with equality if and only if $\mu_H(S)=\mu'$ (by Lemma \ref{lem:monotonicity_of_beta_1} (iv)). Since $h^0(S)=h^0(E)=\beta_{\rk(S),\mu'}$, it must be the case that $\mu_H(S)=\mu'$, and it follows that $S$ is $\mu_H$-semistable (since the slope of any non-zero subsheaf of $S$ must be a rational number that is at most $\mu$, and has denominator belonging to the set $\{2,4,\ldots,2\rk(S)\}$). So $S$ is a maximal sheaf (according to Definition \ref{def:deficiency_maximality}). Theorem \ref{thm:full_structure_thm_maximal_sheaves_with_vanishing_h1} now shows that $S$ is a balanced twisted Steiner-like vector bundle, and $h^1(S)=0$. Thus, $h^0(Q)=0$, where $Q$ is the cokernel of the inclusion $S\hookrightarrow E$.
\end{proof}

Arguing as above, we can characterize semistable sheaves achieving the bound in Corollary \ref{cor:bound_sections_non_gen_gg}. For any such sheaf $E$, the image of the canonical evaluation map must have rank $\rk(E)-1$.

\begin{theorem}\label{thm:non_gen_gg_maximal_small_slope}
	Let $E$ be a $\mu_H$-semistable sheaf of rank $r\geq2$ on $\Quadric$ with $\mu:=\mu_H(E)\in[0,1)$. Suppose $E$ is not generically globally generated. Let $\mu'\in\left[0,\mu\right]$ be the largest rational number that can be written with denominator belonging to the set $\{2,4,\ldots,2(r-1)\}$, and suppose that $$h^0(E)=\beta_{r-1,\mu'}.$$ Then $E$ fits in an exact sequence of the form $$0\to S\to E\to Q\to 0,$$ where $S$ is a balanced twisted Steiner-like vector bundle of rank $r-1$ and slope $\mu'$, and $Q$ is a rank $1$ sheaf with no global sections.
\end{theorem}

We now show that the bound in Theorem \ref{thm:stratified_bound} is sharp.

\begin{proposition}\label{prop:non_gen_gg_maximal_sharpness}
	There exists a $\mu_H$-semistable sheaf $E$ on $\Quadric$ that is not generically globally generated, and that achieves the bound on $h^0(E)$ furnished by Theorem \ref{thm:stratified_bound}.
	
	In other words, there exists a $\mu_H$-semistable sheaf $E$ of rank $r\geq1$ on $\Quadric$ with $h^0(E)>0$ that is not generically globally generated, such that $h^0(E)=\beta_{\rk(S),\mu'}$, where $S\subset E$ is the image of the canonical evaluation map $\OO_{\Quadric}\otimes H^0(E)\to E$, and $\mu'\in\left[0,\mu_H(E)\right]$ is the largest rational number that can be written with denominator belonging to the set $\{2,4,\ldots,2\rk(S)\}$.
\end{proposition}

\begin{proof}
	Due to Theorem \ref{thm-ss_ext_elliptic}, a sheaf $E$ given by a general extension $$0\to S\to E\to Q\to 0$$ satisfies our requirements, where $S$ and $Q$ are general stable sheaves on $\Quadric$ satisfying all of the following properties:\\[-11pt] \begin{enumerate}
		\item $S$ is balanced twisted Steiner-like,\\[-5pt]
		\item $h^0(Q)=0$,\\[-5pt]
		\item $\mu_H(S)$ equals the largest rational number $\mu'\in\left[0,\mu\right]$ that can be written with denominator belonging to the set $\{2,4,\ldots,2\rk(S)\}$, where $$\mu:=\frac{(c_1(S)+c_1(Q))\cdot H}{2(\rk(S)+\rk(Q))},$$
		\item $\chi(Q)$ is \emph{negative enough},\\
		\item $\displaystyle\frac{c_1(S)\cdot(-K_{\Quadric})}{\rk(S)}<\frac{c_1(Q)\cdot(-K_{\Quadric})}{\rk(Q)}$,\\[0pt]
		\item $\rk(S)$ and $c_1(S)\cdot(-K_{\Quadric}))$ are coprime,\\[-5pt]
		\item $\rk(Q)$ and $c_1(Q)\cdot(-K_{\Quadric}))$ are coprime.\\[-5pt]
	\end{enumerate} Note that $E$ is not generically globally generated since $Q$ is not.
\end{proof}

\bibliography{mybib}{}
\bibliographystyle{amsalpha}

\end{document}